\documentclass[]{interact}

\usepackage{epstopdf}% To incorporate .eps illustrations using PDFLaTeX, etc.
\usepackage[caption=false]{subfig}

\usepackage{graphicx}
\usepackage{float}
\restylefloat{figure}
\usepackage{subfig}

\usepackage[numbers,sort&compress]{natbib}% Citation support using natbib.sty
\bibpunct[, ]{[}{]}{,}{n}{,}{,}% Citation support using natbib.sty
\theoremstyle{plain}% Theorem-like structures provided by amsthm.sty
\newtheorem{theorem}{Theorem}[section]
\newtheorem{lemma}[theorem]{Lemma}
\newtheorem{corollary}[theorem]{Corollary}
\newtheorem{proposition}[theorem]{Proposition}
\newtheorem{assumption}[theorem]{Assumption}

\theoremstyle{definition}

\theoremstyle{remark}
\newtheorem{remark}{Remark}

\begin{document}

\title{On Target-Oriented Control of H\'{e}non and Lozi maps}

\author{
\name{E. Braverman\textsuperscript{a}\thanks{CONTACT E. Braverman. Email:
maelena@ucalgary.ca} and  A. Rodkina\textsuperscript{b}}
\affil{\textsuperscript{a}
Department of Mathematics and Statistics, University of Calgary,  
2500 University Drive N.W., Calgary, AB T2N 1N4, Canada;
\textsuperscript{b} 
Department of Mathematics,
The University of the West Indies, Mona Campus, Kingston, Jamaica}
}
%\email{alexandra.rodkina@uwimona.edu.jm}

\maketitle

\begin{abstract}
We explore stabilization for nonlinear systems of difference equations with modified Target-Oriented Control and a chosen equilibrium as a target, both in deterministic and stochastic settings. The influence of stochastic components in the control parameters is explored. 
The results are tested on the H\'{e}non and the Lozi maps.
\end{abstract}


\begin{amscode}
%{\bf AMS classifications:} 
39A11, 39A30, 39A50
\end{amscode}

\section{Introduction}
\label{introduct}

Unstable and chaotic behaviour of dynamical systems gave rise to 
the problem of their stabilization and synchronization actual in many areas of science and engineering.
The idea to apply control designed by the difference of the current and some other, for instance, approximated state, appeared in the seminal paper introducing Ott-Grebogi-Yorke (OGY) method \cite{OGY}. 
The review of chaos control methods both in continuous and discrete cases can be found in \cite{handbook}.
Application of delayed feedback control is going back to the pioneering paper of Pyragas \cite{Pyragas1992} which produced a simple and efficient method to stabilize unstable equilibrium points for a system of differential equations. The method of delayed control was further developed and improved, see the recent papers \cite{Gjur1,Gjur2,Gjur3}
and references therein.  For discrete dynamical systems, application of controls using the difference of the state variable with a computed value at the next step, or after several 
steps, appeared to be more efficient than incorporating delays. This idea, coherent with the OGY method,
is implemented in the Prediction Based Control (PBC) \cite{FL2010,uy99} where the controlled value of the map is a weighted average between the state variable and a computed value which, in the simplest case, evaluates the map at the next step. 
Though there are common ideas in the areas of differential and difference equations, there are some substantial differences in approaches and methods. Here we focus on discrete dynamical systems, requiring for stability of autonomous maps that all the eigenvalues are inside the unit circle (compared to negative real parts for systems of differential equations).

The purpose of the present paper is to investigate deterministic and stochastic chaos control 
for nonlinear system, as well as test and compare them using the examples of
the H\'{e}non and the Lozi maps.  Here we focus on Target-Oriented Control (TOC) with a fixed point chosen to stabilize as a target. TOC with a chosen and fixed target becomes a one-parameter method in the deterministic scalar case. 
In two-dimensional H\'{e}non and Lozi maps, with noise involved in the control, the number of parameters increases, where some are more important than the others, and their interplay is illustrated in examples and simulations.

The H\'{e}non \cite{Henon1976} and the Lozi \cite{Lozi} maps, their generalizations, limit orbits, dynamical properties and various applications continue to attract attention of researchers, see, for example, 
recent publications \cite{add1,add2,add3,add4,add5,add6,add7} and references therein. Unlike these papers, we do not deal with complex dynamics and structure of the attracting sets (though these sets are presented in simulations) but dedicate our efforts to stabilization, illustrating general statements with these two types of maps. The issue of stabilization is of special interest in population dynamics \cite{TPC,Liz2010,FL2010,LP2014}. Another focus is on the influence of stochastic perturbations on stabilization, which also is of significant importance in Mathematical Biology \cite{Chesson, Schreiber1} due to the fact that noise is ubiquitous  in ecological systems, and is sometimes considered to be an important factor in sustaining biodiversity. The idea of stabilization by noise originates in the work of Khasminskii \cite{Kh} and was employed  in many publications during following years. Among other works,  stabilization of difference equations was studied  in \cite{AMR, BR1, BR2019},
stabilization  for certain  higher-dimensional models in \cite{Medv}, and finite difference schemes for two-dimensional linear systems in \cite{BBKR,BK}. The approach of our paper  is closest to \cite{BKR, BKR2020, BR2017}, where prediction-based, proportional feedback and target oriented controls with stochastically perturbed parameters were studied for scalar difference equations.

One of our purposes in this paper is to establish the range of control parameters in TOC and noise amplitudes for which a system without noise is unstable, but addition of noise to the control component leads to stabilization. 

The H\'{e}non system
\begin{equation}
\label{henon}
\begin{array}{ll}
x_{n+1}= & y_n +1 -a x_n^2, 
\\
y_{n+1} = & b x_n,
\end{array}
\end{equation}
for $n \in {\mathbb N}_0 := \{ 0 \} \cup {\mathbb N}$, has two equilibrium points
\begin{equation*}
%\label{henon_eqil}
(x^*,y^*) = \left(  \frac{1}{2a} \left[ b-1 \pm \sqrt{4a + (b-1)^2}\right], bx^* \right), \quad a, b>0.
\end{equation*}
The Lozi map leads to a system of difference equations
\begin{equation}
\label{lozi}
\begin{array}{ll}
x_{n+1}= & y_n +1 -a |x_n|,
\\
y_{n+1} = & b x_n,
\end{array}
\end{equation}
for $n \in {\mathbb N}_0$, with two equilibrium points
\begin{equation*}
%\label{lozi_eqil}
(x^*,y^*) = \left(  \frac{1}{1 \pm a-b}, b x^*  \right), \quad -a<1-b<a, ~~ a,b>0.
\end{equation*}
%whenever $-a<1-b<a$, $a>0$, $b>0$. 

Following the tradition, we consider $a=1.4$, $b=0.3$ in most of examples, but also discuss the influence of other values for $a$ and $b$. 
It is well known that for the values $a=1.4$, $b=0.3$, both 
\eqref{henon} and \eqref{lozi} are chaotic. Looking at the bifurcation diagram on 
Fig.~\ref{figure_ex_1_fig_1} (a), with $\alpha=0$, we see that 
the limit set of the uncontrolled H\'{e}non map includes the segment $[-1.2,1.2]$, no stability is observed.
Similar chaotic properties of the Lozi map can be observed on Fig.~\ref{figure_ex_1_fig_2} (a), for $\alpha =0$.

%in-text figure
\begin{figure}
%\centering
\subfloat[]{%
\resizebox*{4cm}{!}{\includegraphics{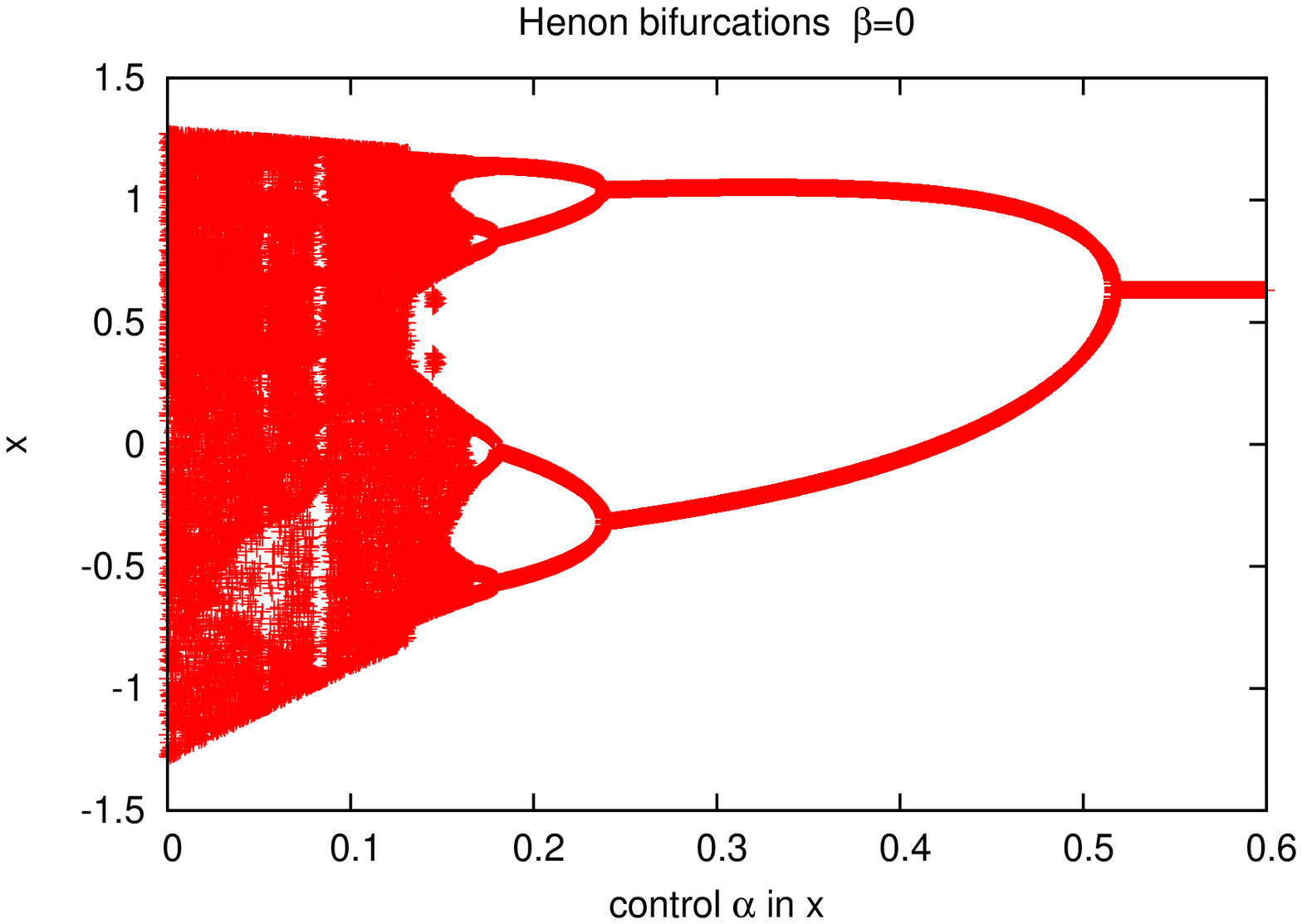}}}
\hspace{15pt}
~~~~~~~~~~~~~~~~~
\subfloat[]{%
\resizebox*{4cm}{!}{\includegraphics{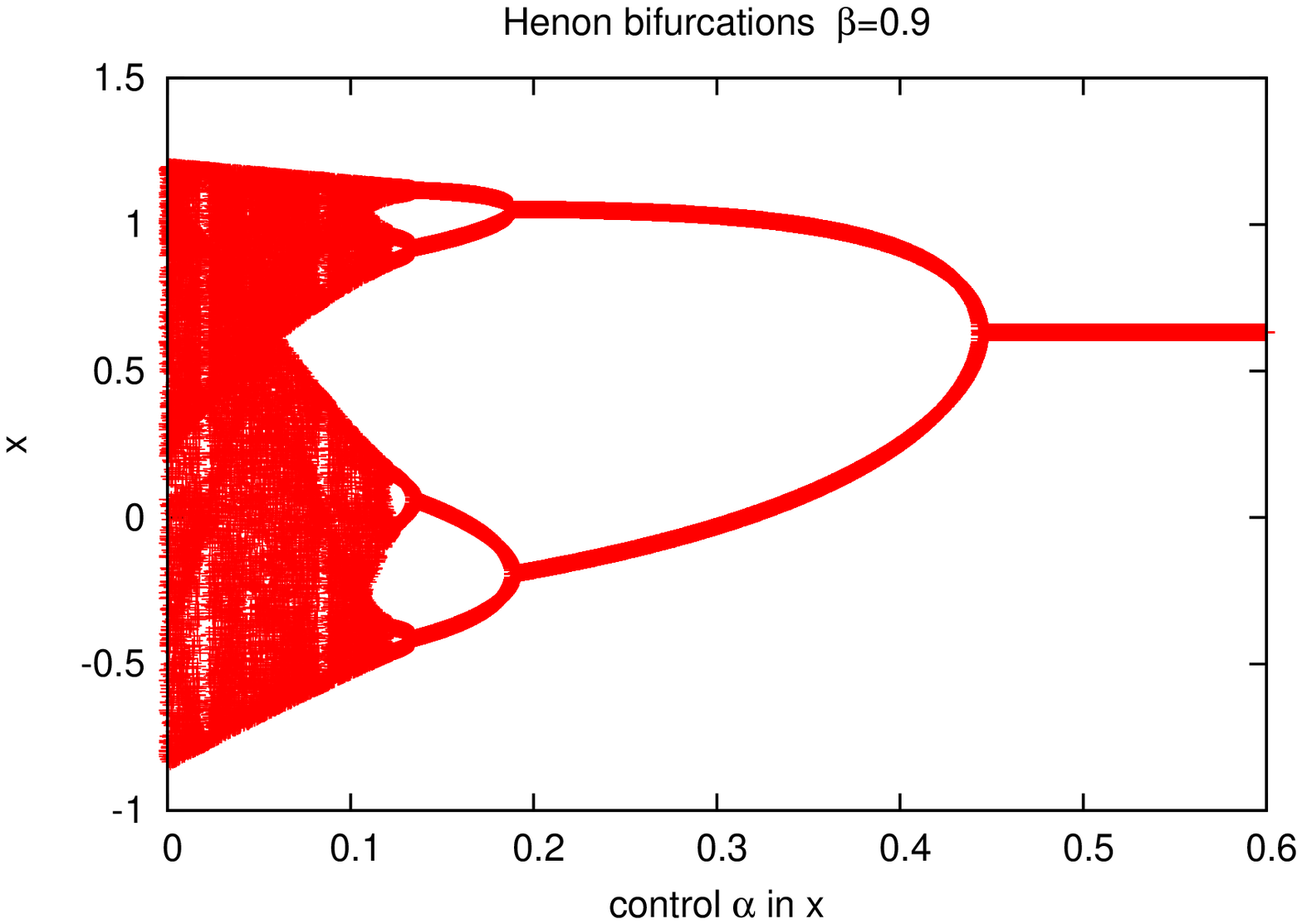}}}
\caption{Bifurcation diagram of the H\'{e}non map for $x$ and (a) no control in $y$, (b) 
$\beta =0.9$ control in $y$. The range of initial values is $x_0 \in [0.1,0.8]$ and $y_0 \in [0.1,0.2]$.
\bigskip
}
\label{figure_ex_1_fig_1}  
\end{figure}

\bigskip
\bigskip

%in-text figure
\begin{figure}
\bigskip
\bigskip
%\centering
\subfloat[]{%
\resizebox*{4cm}{!}{\includegraphics{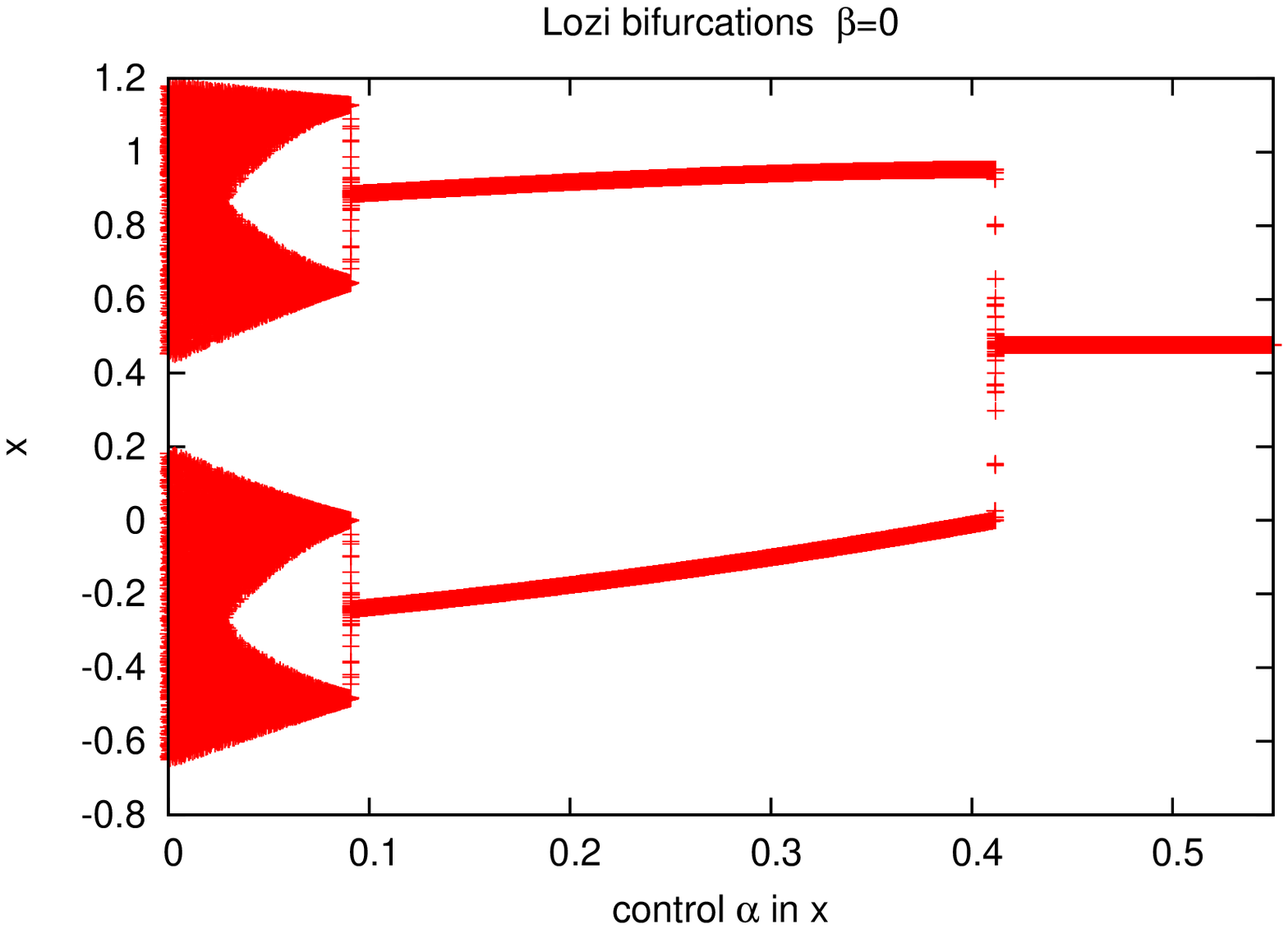}}}
\hspace{15pt}
~~~~~~~~~~~~~~~~~
\subfloat[]{%
\resizebox*{4cm}{!}{\includegraphics{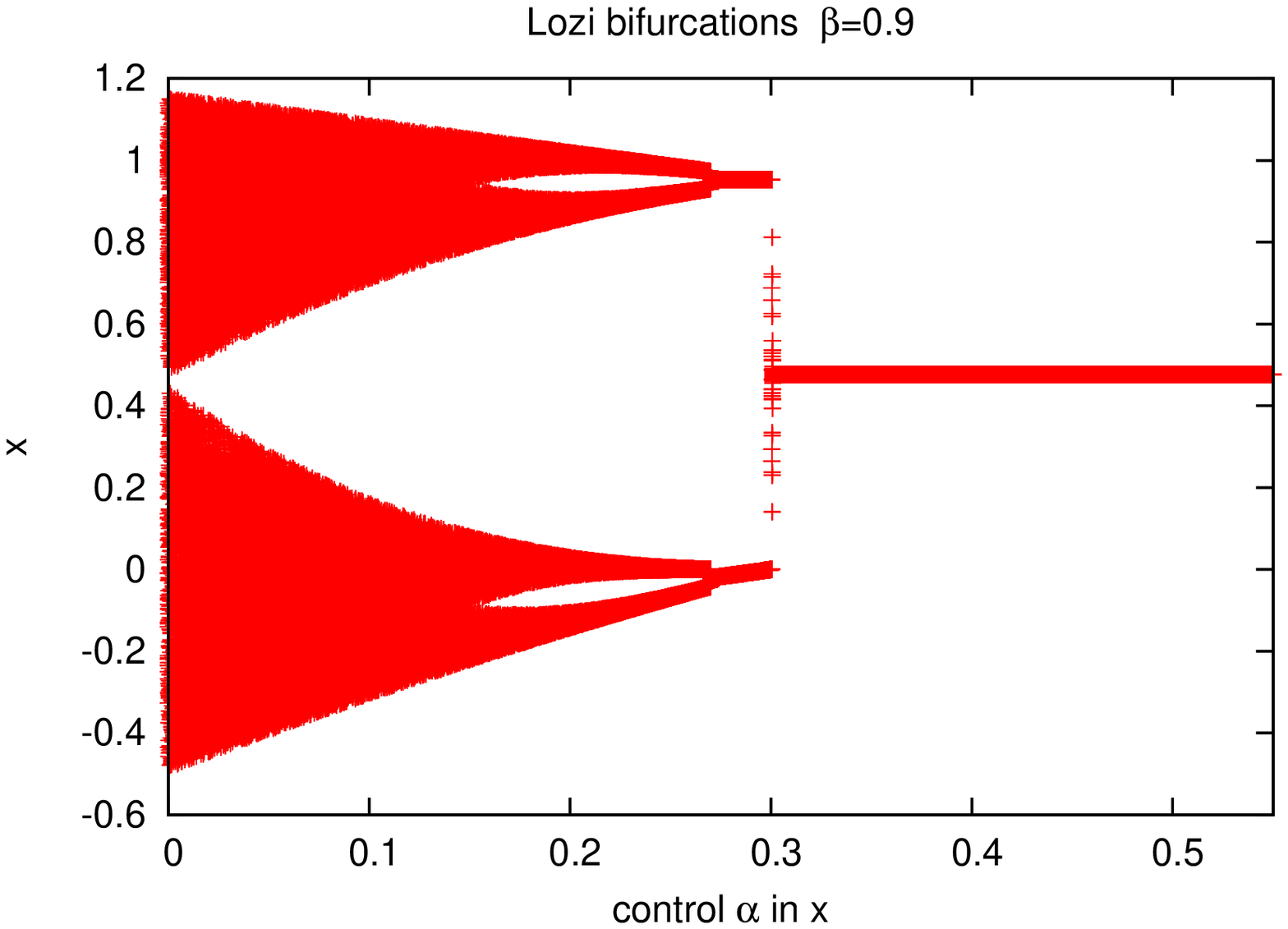}}}
\caption{Bifurcation diagram of the Lozi map for $x$ and (a) no control in $y$, 
(b) $\beta =0.9$ control in $y$. The range of initial values is $x_0 \in [0.1,0.8]$ and $y_0 \in [0.1,0.2]$.
\bigskip
}
\label{figure_ex_1_fig_2}  
\end{figure}

Both systems can be described as a nonlinear autonomous system
\begin{equation*}
%\label{eq:intr1}
X_{n+1} = F(X_n), \quad X_n \in {\mathbb R}^m,~~n \in {\mathbb N}_0.  %=\mathbb N\cup \{0\}.
\end{equation*}

Scalar Target-Oriented Control (TOC) introduced first in \cite{Dattani} establishes a target state $T$ and implements an increase or a decrease of the state variable each time step, 
depending on  whether its value exceeds or is below the target state.
For a one-dimensional equation $x_{n+1}=f(x_n)$ with an arbitrary continuous function $f: {\mathbb R} \to {\mathbb R}$, TOC has the form
\begin{equation}
\label{TOC_eq}
x_{n+1}=f( c T+(1-c)x_n ), \quad T\ge 0, \; c\in [0,1).
\end{equation}
Here $c$ is the control intensity, for $c=0$ we get the original map, and, as $c\to 1$, the right-hand side in \eqref{TOC_eq} approaches $f(T)$. 
Once the target $T$ is chosen as a fixed point of $f$, obviously it is 
also a fixed point of \eqref{TOC_eq}, and TOC becomes an efficient method of its stabilization. If $T=0$ we get an earlier known Proportional Feedback (PF) control method \cite{gm} with a proportional reduction of the state variable at each step. PF control of scalar maps was studied in detail in \cite{Carmona,Liz2010}. Under certain assumptions, it stabilizes the zero equilibrium which is important in pest control problems and is aligned with TOC stabilization using a stabilized fixed point as a target. In other cases, a non-zero equilibrium can be stabilized.
For an arbitrary $T$, a fixed point of \eqref{TOC_eq} is a weighted average of a positive fixed point of $f$ \cite{TPC}, assuming it exists and is unique, and $T$, with the weight depending on the control intensity $c$.

In \cite{BF2017}, we considered vector extensions of TOC for systems  (VTOC)
\begin{equation*}
%\label{TOC}
X_{n+1}= F(cT +(1-c)X_n), \quad 
%T \in {\mathcal D}, \; 
c\in [0,1),
\end{equation*}
as well as Vector Modified Target Oriented Control (VMTOC)
\begin{equation}
\label{MTOC}
X_{n+1} = c T +(1-c) F(X_n), \quad %T \in B_R(X^*), \; 
c\in [0,1).
\end{equation}

In this paper we apply VMTOC to \eqref{henon} and \eqref{lozi} with $F=F_1$ and $F=F_2$, 
respectively, where 
$$
F_1((x,y)') = (y+1-ax^2,bx)', \quad F_2((x,y)') = (y+1-a|x|,bx)',
$$
we use $ \cdot '$ for transposition.

In addition to VMTOC \cite{BF2017} with a constant control intensity \eqref{MTOC},
we describe some important modifications.
\begin{itemize}
\item
We consider variable control intensity. 
For TOC, a point being stabilized is a weighted average of the target and a fixed point of the original map.
Thus, to stabilize a chosen equilibrium, we have to choose this fixed point as a target.
\item
Using the results of the previous item, a stochastic component of the control can be introduced and studied, with the noise amplitude $\ell$. Moreover, with $\ell>0$ stabilization can be achieved in the cases, when for $\ell =0$, a two-cycle is stable rather than the target equilibrium. 
\item
In \eqref{MTOC}, %{henon_c} and \eqref{lozi_c},  
the control is implemented using left multiplication by a scalar matrix $cI$. We consider more general matrices. The choice of diagonal matrices gives a good trade-off of simplicity and differential approach to the variables. 
For both H\'{e}non and Lozi maps, $x$ is more `sensitive' to control than $y$, as expected. 
\end{itemize}

Chaos control of nonlinear applied systems is a more challenging problem than for scalar maps. 
The present paper considers VMTOC \cite{Chan2014,BF2015,BF2017,Dattani,TPC} with a chosen equilibrium point for the H\'{e}non and the Lozi maps.
The theory applies to any of the two points, however, in examples and simulations we reduce ourselves to stabilization of $\displaystyle (x^*,y^*)=\left(  \frac{1}{2a} \left[ b-1 + \sqrt{4a + (b-1)^2}\right], bx^* \right)$ for the H\'{e}non and $\displaystyle (x^*,y^*)=\left(  \frac{1}{1 + a-b}, b x^*  \right)$ for the Lozi map, respectively, using it as a target.

Constant, variable and stochastic control types are investigated. For constant control intensity, estimation of convergence with an initial point in a neighborhood of the target is considered. 
First, the influence of the control intensity for $x$ is much more significant than contribution of control in $y$.  Second, while theoretically control level should increase with the increase of the domain for initial values, in simulations we see that the control value providing 
local stabilization will work for a larger domain in a certain neighbourhood of the chosen equilibrium point.
Next, variable control is important in itself, describing a possibility of changing
intensity, as well as for further application to a stochastically perturbed control.  
Introduction of stochastic control is two-fold: 
\begin{enumerate}
\item
to demonstrate the range of noise which keeps stability for the same interval of parameters and the initial domain, where a deterministic system with variable deterministic control stabilizes a chosen equilibrium; 
\item
to improve deterministic results in a sense that there is stability in a stochastic case, while no stability in the deterministic case with the mean control values is observed (compare with \cite{BR2019,Medv}).
\end{enumerate}

For the first part, results for variable coefficients provide a required background, the same method was applied in \cite{BKR}.
For the second part, sometimes a neighborhood of the equilibrium point in which stabilization can be achieved is quite small \cite{BR2019}. However, here we get robust results (illustrated with simulations and bifurcation diagrams) for the H\'{e}non map and a global type estimate for the Lozi map with more restrictions on the control parameters than in the right half-plane.

The rest of the paper has the following structure.
In Section~\ref{sec_main} we consider stabilization of \eqref{henon} and \eqref{lozi} with a modified deterministic version of TOC. In Section~\ref{sec:stoch}, a control is assumed to be perturbed by some noise. In both sections, theoretical results are illustrated with examples.
A brief discussion of the results and comparison to other control methods in  Section~\ref{sec:conclusions} conclude the paper.

%%%%%%%%%%%%%%%%%%%%%%

\section{Stabilization of Equilibrium Points by VMTOC}
\label{sec_main}

\subsection{General deterministic step-dependent control}

We generalize VMTOC \eqref{MTOC} in ${\mathbb R}^m$ to a system with a matrix control
\begin{equation}
\label{eq:VMTOC}
X_{n+1}= U T+(I-U) F(X_n), \quad X_0 \in B_R(X^*), \quad  n\in {\mathbb N}_0,
\end{equation}
where $X_0\in {\mathbb R}^m$, $U$ is a constant $m \times m$ matrix, $I$ is the identity matrix, 
a certain norm $\| \cdot \|$ is assumed in $\mathbb R^m$, the same notation will be used for an induced matrix norm,
$$
B_R(X^*) := \left\{ X\in {\mathbb R}^m: \| X-X^* \| < R \right\}
$$
is an open ball.
For stabilization of a chosen equilibrium $X^*$, we use it as a target $T=X^*$ and take $U$ as a diagonal matrix such that $U=$diag$(d_1, \dots, d_m)$.

In addition to constant diagonal $U$ in \eqref{eq:VMTOC}, we consider step-dependent 
VMTOC with an infinite family of diagonal matrices $\{ U_n \}$ with the entries $d_{1,n},\dots, d_{m,n} \in [0,1)$ on the diagonal
\begin{equation}
\label{eq:VMTOC_var}
X_{n+1} = U_{n+1} X^* +(I-U_{n+1}) F(X_n), \quad X_0 \in B_R(X^*), \quad n\in {\mathbb N}_0.
\end{equation}
Then, VMTOC with different matrices at various steps becomes
\begin{equation*}
\label{eq:VMTOC_var_a}
X_{n+1}= F(X_n) - U_{n+1}(F(X_n)-X^*), ~ X_0 \in B_R(X^*), ~ n\in {\mathbb N}_0,
\end{equation*} 
with $F(X^*)= X^*$, leading to
$$
X_{n+1} - X^* = (I - U_{n+1})(F(X_n) - X^*).
$$

\begin{theorem}
\label{toc_lemma_var}
Assume that there exist $R>0$, the diagonal matrices $U_n$ and $\nu^* = \nu^*(R) \in (0,1)$ such that 
\begin{equation}
\label{eq:diff_cond_var}
\| (I-U_{n}) [F(X) - X^* ]\| \leq  \nu^* \|X - X^*\|, \quad   
n\in {\mathbb N}, \quad  X \in B_R(X^*).
\end{equation}

Then any solution $X_n$ of \eqref{eq:VMTOC_var} with $X_0 \in B_R(X^*)$ converges to $X^*$, and $X^*$ is asymptotically stable in $B_R(X^*)$. 
\end{theorem}

\begin{proof}
If $X_0 \in B_R(X^*)$, we have $\| X_0-X^* \| < R$ 
and
\begin{equation*}
\begin{split}
\| X_1 - X^* \| = & \| U_1 X^* +(I-U_1) F(X_0) - X^* \| \\ =  & \|  (I-U_1) [F(X_0) - X^* ] \| \leq  \nu^* \|X_0 - X^*\| < R.
\end{split}
\end{equation*}
Using induction, assume that $X_n \in B_R(X^*)$ and 
$ \| X_n - X^* \|  \leq  (\nu^*)^n \|X_0 - X^*\|$. Then,
by \eqref{eq:diff_cond_var},
\begin{equation*}
\begin{split}
\| X_{n+1} - X^* \|  & = \| U_{n+1} X^* +(I-U_{n+1}) F(X_n) - X^* \| 
\\ & = \| (I-U_{n+1}) [F(X_n) - X^* ] \| \leq   \nu^* \| X_n-X^* \|  \leq   (\nu^*)^{n+1} \| X_0-X^* \| , 
\end{split}
\end{equation*}  
which implies that all $X_n \in B_R(X^*)$ and $\| X_n - X^* \|   \leq   (\nu^*)^n \| X_0 - X^* \| \to 0,$ as $n \to \infty$, for any $X_0 \in B_R(X^*)$. So we get convergence to $X^*$, with a guaranteed convergence rate, thus $X^*$ is asymptotically stable in $B_R(X^*)$.
\end{proof}

Further, Theorem~\ref{toc_lemma_var} allows to stabilize any locally Lipschitz map in a chosen domain. 

\begin{corollary}
\label{cor:toc_determ_var}
Suppose that there are constants $R>0$ and $L>0$ such that
\begin{equation}
\label{eq:Lipschitz_gen}
\| F(X)-X^* \| \leq L \| X- X^* \|, \quad  X \in B_R(X^*).
\end{equation}
Then, for each $\nu^* \in (0,1)$, there exists a diagonal matrix $U=${\rm diag}$(d_1,\dots, d_m)$ with the entries $d_1,\dots, d_m\in [0,1)$ such that \eqref{eq:diff_cond_var} holds, and $X^*$ is an asymptotically stable point of \eqref{eq:VMTOC_var} in $B_R(X^*)$.
%{\bf???(AR: Since in the proof we consider case  $U_n=$diag$(d_{1,n}, \dots, d_{m,n})$, it is logical to %include this case into the statement, like it is done in Corollary 2. Or do not talk about it in the %proof. )???}
\end{corollary}

\begin{proof}
For $\displaystyle U= \left( 1 - \frac{\nu^*}{\max\{1,L \| I \|\}}  \right) I$, we get
\begin{equation*}
\begin{split}
\| (I-U) [F(X) - X^* ]\| \leq & \| I-U \| \| F(X) - X^* \| \\ \leq &  \frac{\nu^*}{\max\{1,L \| I \| \}} L  \| X- X^* \| 
 \leq \nu^*
\| X- X^* \|.
\end{split}
\end{equation*}
For $U_n=$diag$(d_{1,n}, \dots, d_{m,n})$ with $\displaystyle 1 - \frac{\nu^*}{\max\{1,L \| I \|\}} \leq d_{j,n} < 1$, we also get 
$$
\| (I-U_{n}) [F(X) - X^* ]\| \leq  \| I-U_{n} \| \| F(X) - X^* \| \leq \nu^* \| X- X^* \|,
$$
so \eqref{eq:diff_cond_var} holds, and reference to Theorem~\ref{toc_lemma_var} concludes the proof.
\end{proof}

Sharper control estimates with $U$ different from a scalar matrix determined in the proof of Corollary~\ref{cor:toc_determ_var}, are possible for specific $F$ and chosen norms.
A particular case of $F$ and $\|\cdot\|_{\infty}$ norm,  $ \displaystyle \| (x_1, \dots, x_m)\|_{\infty} = \max_{k=1, \dots, m} |x_k|$, is outlined below. 
\begin{corollary}
\label{th:toc_determ_var}
Let $F=(f_1,\dots,f_m)'$, $X=(x_1, \dots, x_m)'$, $T=X^*=(x_1^*, \dots, x_m^*)'$ be an equilibrium point, $R>0$ and positive constants $A_{ij}$, $i,j = 1, \dots, m$, ${\mathcal A} = (A_{ij})_{i,j=1}^m$ be such that 
\begin{equation}
\label{eq:Lipschitz}
| f_i(X)-x_i^* | \leq \sum_{j=1}^m A_{ij} \left| x_j - x_j^* \right|, \quad  X \in B_R(X^*).
\end{equation}
Then, for each $\nu^* \in (0,1)$ there exists $U=$diag$(d_1, \dots, d_m)$ with $d_1,\dots, d_m\in [0,1)$  such that in the induced norm $\| \cdot \|_{\infty}$, $\| (I-U) {\mathcal A} \| \leq  \nu^*$.

Also, for any diagonal matrices $U_n=${\rm diag}$(d_{1,n}, \dots, d_{m,n})$ with $d_j \leq d_{j,n} <1$, $j=1, \dots, m$, 
$X^*$ is an asymptotically stable point of \eqref{eq:VMTOC_var} in $B_R(X^*)$.
\end{corollary}

\begin{proof}
Let \eqref{eq:Lipschitz} hold. Denote 
\begin{equation}
\label{eq:Lip_const}
L_i: = \sum_{j=1}^m A_{ij}, \quad L := \max_{i=1, \dots, m} L_i.
\end{equation}
Let us fix any $\nu^* \in (0,1)$ and define $d_i=\max\left\{1- \frac{\nu^*}{L_i}, 0\right\}$.
Then, $(1-d_i)L_i\leq \nu^*$, for each $i=1, \dots, m$, which implies $\| (I-U) {\mathcal A} \| \leq  \nu^*$. 
 
Consider now any diagonal matrix $U_n=$diag$(d_{1,n}, \dots, d_{m,n})$ with $d_j \leq d_{j,n} <1$, $j=1, \dots, m$.
Then the $i$-th entry of $ (I-U_n) [F(X) - X^* ]$ satisfies
\begin{equation*}
\begin{split}
 & |(1-d_{i,n})(f_i(X)-X^*) |  \leq (1-d_{i,n}) \sum_{j=1}^m A_{ij} 
\left| x_{j} - x_j^* \right|  \\
\leq &  (1-d_i)L_i \max_{j=1, \dots, m}  \left| x_{j} - x_{j}^* \right|  \leq \nu^* \max_{j=1, \dots, m}  \left| x_{j} - x_{j}^* \right|.
\end{split}
\end{equation*}
Thus \eqref{eq:diff_cond_var} holds, and the reference to Theorem~\ref{toc_lemma_var} concludes the proof.
\end{proof}

\begin{remark}
\label{rem:linearization}
Assume that all $f_i$ are continuously differentiable. Then in \eqref{eq:Lipschitz}, by the Mean Value Theorem, for some $\bar{X} \in B_R(X^*)$,
$$
| f_i(X)-x_i^* | = \left| \nabla f(\bar{X}) (X-X^*) \right| \leq 
\sum_{j=1}^m \max_{ (x_1, \dots, x_m) \in B_R(X^*)} \left| \frac{\partial f_i}{\partial x_j} \right| \left| x_j - x_j^* \right|.
$$
Thus, as $R$ in \eqref{eq:Lipschitz} becomes small,  for a smooth map $F$, 
the values of $A_{ij}$ in \eqref{eq:Lipschitz} approach $\displaystyle \left| \frac{\partial f_i}{\partial x_j} (X^*) \right|$.
\end{remark}

In Sections~\ref{sec: step_independent}-\ref{subsec:varHL}, we  consider examples of the H\'{e}non map,  $a=1.4$, $b=0.3$, with the equilibrium $(x^*,y^*) \approx (0.6314,0.1894)$ and of the Lozi map,  $a=1.4$, $b=0.3$, with  $(x^*,y^*) \approx (0.4762,0.1429)$.

%%%%%%%%%%%%%%%
\subsection{Step-independent  H\'{e}non and Lozi maps}
\label{sec: step_independent}

First, we illustrate that stabilization with constant VMTOC and the diagonal matrix $U={\rm diag}(\alpha,\beta)$ is possible. The computation is supported with numerical simulations.
Everywhere in simulations, we choose stabilization of the equilibrium in the right half-plane with $|x|=x$, and the initial values satisfying $x_0>0$.   

\subsubsection{\bf H\'{e}non}

Consider constant VMTOC  of  the H\'{e}non map
\begin{equation}
\label{henon_toc}
\begin{array}{ll}
x_{n+1}= & \alpha x^* + (1-\alpha)(y_n +1 -a x_n^2),
\\
y_{n+1} = & \beta y^* + (1-\beta) b x_n,
\end{array}
\end{equation}
where the matrix of the controlled autonomous system is 
$$
(I-U){\mathcal A} = \left[  \begin{array}{cc}  -2(1-\alpha)a x^* & 1- \alpha \\
(1-\beta)b & 0    \end{array} \right]
$$
with two constant values $\beta =0$ and $\beta = 0.9$, to outline the role of $\beta$.  For $\beta=0$,
all the eigenvalues of $(I-U)\mathcal A$ are in the unit circle if and only if \cite[P. 188]{Saber} the trace and the determinant satisfy
\begin{equation}
\label{stab_cond}
| {\rm tr} ((I-U){\mathcal A}) | < 1 + {\rm det} ((I-U){\mathcal A}) < 2.
\end{equation}
Note that,  if both eigenvalues of $(1-U){\mathcal A}$ are in the unit circle, there exists a norm $\|\cdot\|$ such that 
 $\|(1-U){\mathcal A}\|<1$. 

Since ${\rm det} ((I-U){\mathcal A})<0$, the right inequality in \eqref{stab_cond} is valid, and the left one is equivalent to 
\begin{equation*}
\begin{split}
2(1-\alpha)a x^* < &  1 - (1-\alpha) b \iff 1-\alpha < \frac{1}{2ax^*+b} \\ \approx & \frac{1}{2.8 \cdot 0.631354+0.3} = 1-\alpha^* \approx 0.48361.
\end{split}
\end{equation*}
Thus, for $\alpha > \alpha^* \approx 0.51639$, the equilibrium $(x^*,y^*)\approx (0.6314,0.1894)$ of \eqref{henon_toc} is locally asymptotically stable.
In Fig.~\ref{figure_ex_1_fig_1} (a), we can observe in the bifurcation diagram that the last period-halving bifurcation is quite close 
to the theoretically predicted value for local stability and is $\alpha^* \approx 0.5164$.

Further, assume that some $\alpha> \alpha^* \approx 0.51639$ is set and notice that
$\displaystyle \frac{1}{2a(1-\alpha)} - x^*>0$ due to the choice of $\alpha> \alpha^*$. Then, for any 
$\displaystyle R \in \left( 0, \frac{1}{2a(1-\alpha)} - x^* \right)$ and $(x_0,y_0) \in B_R((x^*,y^*))$, the solution tends to the equilibrium in the first quadrant.

For the bifurcation diagram in Fig.~\ref{figure_ex_1_fig_1} (a), we took quite a large $R > 0.5$.

For $\beta=0.9$, all the eigenvalues of $\displaystyle (I-U){\mathcal A} = \left[  \begin{array}{cc}  -2(1-\alpha)a x^* & 1- \alpha \\
0.1 b & 0    \end{array} \right]$ satisfy 
$|\lambda| <1$ for $1-\alpha < 1/( 2ax^*+ 0.1 b)  \approx 0.55624$. 
Therefore for $\beta=0.9$ and $\alpha > \alpha^* \approx 0.44376$, the positive equilibrium $(x^*,y^*)$ of \eqref{henon_toc} is locally asymptotically stable. 
In Fig.~\ref{figure_ex_1_fig_1} (b), we see that for $\beta =0.9$, the last period-halving bifurcation with stability for larger values is at 
$\alpha \approx 0.444$, which is aligned with predicted local stability.

%%%%%%%%%%%%%%%%%%%

\subsubsection{\bf Lozi}
%\begin{example}
%\label{ex:lozi1}

Consider constant VMTOC  of  the  Lozi map
\begin{equation}
\label{lozi_toc}
\begin{array}{ll}
x_{n+1}= & \alpha x^* + (1-\alpha)(y_n +1 -a |x_n|), 
\\
y_{n+1} = & \beta y^* + (1-\beta) b x_n,
\end{array}
\end{equation}
with $x_0,x_n>0$ and $(I-U){\mathcal A}$ being the Jacobian matrix of controlled system \eqref{lozi_toc}
$$
(I-U){\mathcal A} = \left[  \begin{array}{rr}  -(1-\alpha)a  & 1- \alpha \\
(1-\beta)b & 0    \end{array} \right].
$$

For $\beta=0$, conditions \eqref{stab_cond} are equivalent to
$$
(1-\alpha) a < 1 - (1-\alpha) b \iff 1-\alpha < \frac{1}{a+b} \approx 0.588235,
$$
and for $\alpha > \alpha^* \approx 0.411765$, the positive equilibrium of \eqref{lozi_toc} is locally asymptotically stable. 
In Fig.~\ref{figure_ex_1_fig_2} (a), we see that the last period-halving bifurcation is at $\alpha \approx 0.412$, as predicted.

Similarly, once $\beta=0.9$, the inequality $1-\alpha < \frac{1}{a+0.1b} \approx 0.6993$ holds for
 $\alpha > \alpha^* \approx 0.3007$ which is illustrated with Fig.~\ref{figure_ex_1_fig_2} (b), where stability starts at $\alpha \approx 0.301$.

\begin{remark}
\label{rem:absLozi}
Assume now that $x^*>0$ and we do not restrict ourselves to the right half-plane.  
For $x_n<0$ in ${\bf l}_{\infty}$-norm, for example, we have 
\begin{equation*}
\begin{split}
\|X_{n+1}-X^*\| & =\left\|\left[  \begin{array}{l}  (1-\alpha)a(|x_n| - x^*)  + (1- \alpha)(y_n-y^*)  \\
(1-\beta)b(x_n-x^*)     \end{array} \right] \right\|\\
%&=\max\{(1-\alpha)a|x_n+x^*| +(1-\alpha)|y_n-y^*|, (1-\beta)b|x_n+x^*|\}\\
&\leq \max\{(1-\alpha)a|x_n-x^*| +(1-\alpha)|y_n-y^*|,(1-\beta)b|x_n-x^*|\} 
\\
& \le \max \{(1-\alpha)(1+a), (1-\beta)b \} \|X_n-X^*\|.
\end{split}
\end{equation*}
Therefore for $b=0.3 \in (0,1)$, $a=1.4$,  we have $(1-\beta)b<0.3$, so  $\nu^*=(1-\alpha)(1+a)<1$ if $\alpha>1-\frac 1{1+a}=0.5833$, which is a larger lower estimate than above.  So stability in all the plane is guaranteed if $\alpha>0.5833$. We later illustrate this fact with simulations for controls with noise.
\end{remark}

%%%%%%%%%%%%%

\subsection{Variable H\'{e}non and Lozi maps}
\label{subsec:varHL}

Further, we proceed to the variable VMTOC and specify constants in the conditions of Corollary~\ref{th:toc_determ_var} for the H\'{e}non and the Lozi maps. Let 
$$
U_n = \left[ \begin{array}{cc} d_{1,n} & 0 \\ 0 & d_{2,n} \end{array} \right], ~~ d_{1,n},d_{2,n} \in [0,1), ~~n \in {\mathbb N}.
$$
This leads to VMTOC of the H\'{e}non map 
\begin{equation}
\label{henon_toc_diag}
\begin{array}{ll}
x_{n+1}= &  d_{1,n+1} x^* + \left(1-d_{1,n+1} \right)
(y_n +1 -a x_n^2),
\\
y_{n+1} = & d_{2,n+1}  y^* + \left(  1-d_{2,n+1} \right)b x_n,
\end{array}
\end{equation}
and the Lozi map
\begin{equation}
\label{lozi_toc_diag}
\begin{array}{ll}
x_{n+1}= & d_{1,n+1}  x^* + \left(1-d_{1,n+1} \right)(y_n +1 -a |x_n|), 
\\
y_{n+1} = &  d_{2,n+1}  y^* + \left(  1-d_{2,n+1}  \right) b x_n.
\end{array}
\end{equation}

For the sake of brevity in the following two  sections we denote 
$\alpha_n=d_{1,n+1}$, $\beta_n=d_{2,n+1}$. 

%%%%%%%%%%%%%%%%%%%%%%

\subsubsection{\bf H\'{e}non}
%{Variable H\'{e}non   map}
\label{subsubsec:var_Henon}

We have
\begin{equation*}
%\label{def:cal_A}
\begin{split}
\left[ \begin{array}{c} 1-ax^2+y-x^* \\ bx-y^* \end{array} \right] & = \left[ \begin{array}{c} (1-ax^2+y)- (1-a(x^*)^2+y^*) \\ b(x-x^*) \end{array} \right] \\ &  =
\left[ \begin{array}{cc} -a(x+ x^*)  & 1 \\ b & 0 \end{array} \right] \left[ \begin{array}{c} x - x^* \\ y-y^* \end{array} \right] =: {\bar A} (X-X^*),
\end{split}
\end{equation*}
where $x \in (x^*-R, x^*+R)$ and $R>0$ is a fixed constant. 
Thus  \eqref{eq:Lipschitz}-\eqref{eq:Lip_const} hold with
\begin{equation}
\label{def:cal_A_coef}
A_{11} =  \max_{x \in [x^*-R, x^*+R]} a|x+ x^*|,~~ A_{12} = 1, ~~ A_{21}= b,~~ A_{22}=0, ~~
{\mathcal A}=(A_{ij})_{i,j=1}^2,
\end{equation}
$L_1 = \max_{x \in [x^*-R, x^*+R]} a|x+ x^*| + 1$ and $L_2=b$. Note that $L_2<1$, for $b=0.3$. Let us follow the computations in the proof of Corollary~\ref{th:toc_determ_var} to find possible diagonal entries of $U_n$ guaranteeing that the norms of the matrices
\begin{equation*}
\begin{split}
(I-U_n){\mathcal A} = & \left[ \begin{array}{cc} \displaystyle (1-\alpha_n) \max_{x \in [x^*-R, x^*+R]} a|x+ x^*| & 1-\alpha_n
 \vspace{2mm} \\ (1-\beta_n) b & 0 \end{array} \right]  \\ = &
\left[ \begin{array}{cc} \displaystyle (1-\alpha_n)a(2x^*+R) & 1-\alpha_n \\ (1-\beta_n) b & 0 \end{array} \right]
\end{split}
\end{equation*}
do not exceed $\nu^*<1$. 
Let us illustrate the dependency of the lower bound on $\alpha_n$ on the choice of the matrix norm.
In ${\mathbf l}_{\infty}$-norm, the numbers $\beta_n \in [0,1)$ can be arbitrary as $b =0.3 \in (0,1)$.  
Consider, for example, a particular case of no control in $y$ (all $\beta_n=0$) and distinguish between the two cases:
\begin{enumerate}
\item
small $R \approx 0.01$. Then $0< x+ x^* < 1.273$. Let us fix $\displaystyle \alpha_* > 1- \frac{1}{a(2x^*+R)+1} \approx 0.641$. For $\alpha_n \in (\alpha_*,1)$ and $x_0 \in (x^*-0.01,x^*+0.01)$, local stabilization is achieved.  
\item
larger $R \approx 0.36$. Then $0< x+ x^* < 1.63$. We fix $\alpha_* \in (0.694,1)$
and get stabilization for all $\alpha_n \in [\alpha_*,1)$.
\end{enumerate} 
Recall that in Fig.~\ref{figure_ex_1_fig_1} (a), we observe stabilization for no control in $y$ and any $\alpha \approx 0.516$ or higher. 
To illustrate the role of $y$-control and sensitivity of the estimates to the choice of a norm, we note that, while in ${\mathbf l}_{\infty}$-norm, the sum of the moduli of the entries in each row of
$(I-U_n){\mathcal A}$ should be less than $\nu^*<1$, $\nu^* \approx 1$, in ${\mathbf l}_1$-norm $\| (x_1, \dots, x_m)' \|_1$,
the required condition is
$$
(1-\alpha_n) a(2|x^*|+R) + (1-\beta_n) b<\nu^*, \quad 1-\alpha_n < \nu^*.
$$ 
The second inequality holds for any $\alpha_n \in [\alpha_*,1)$, while the first one for a small $R \approx 0.01$
and $\beta_n =0$ becomes $\alpha_n \geq \alpha^* > 1 - 0.7/(2ax^*) \approx 0.6041$,
which is a better estimate than in ${\mathbf l}_{\infty}$-norm, while for $\beta_n=0.9$, we get stabilization for $\alpha_n \in (0.4513,1)$. Finally, for ${\mathbf l}_{2}$ vector norm, the induced norm is a spectral norm, with $\| B \|$ for a matrix $B$ being the largest value  $\sqrt{\lambda}$, where $\lambda$ is an eigenvalue of $B'B$.    
For a fixed $\nu^*<1$, $\nu^* \to 1^-$, both eigenvalues $\lambda_i$ of 
\begin{equation*}
\begin{split}
M_n = & ((I-U_n){\mathcal A})'((I-U_n){\mathcal A}) \\ = & 
\left[ \begin{array}{cc} (1-\alpha_n)^2a^2(2x^*+R)^2 + (1-\beta_n)^2 b^2 & (1-\alpha_n)^2 a (2x^*+R) \\ (1-\alpha_n)^2 a (2x^*+R) & (1-\alpha_n)^2 \end{array} \right]
\end{split}
\end{equation*}
should satisfy (the eigenvalues are real and non-negative) $\lambda_i \leq (\nu^*)^2<1$.

It is sufficient to check for the larger eigenvalue $\displaystyle {\rm tr}(M_n) + \sqrt{ {\rm tr}^2(M_n) - 4 {\rm det~}(M_n)} < 2  (\nu^*)^2$, which, under ${\rm tr}(M_n) \leq 2 (\nu^*)^2 < 2$, is equivalent to 
$ {\rm tr}(M_n) \leq (\nu^*)^2 + {\rm det~}M_n/(\nu^*)^2,$
or 
\begin{equation}
\label{est:iv}
(1-\alpha_n)^2a^2(2x^*+R)^2 + (1-\beta_n)^2 b^2 + (1-\alpha_n)^2 \leq (\nu^*)^2 + (1-\beta_n)^2 b^2 (1-\alpha_n)^2(\nu^*)^{-2} < 2.
\end{equation}
Since $\nu^*<1$, for $\nu^* \to 1^-$, we have $(1-\beta_n)^2 b^2 (1-\alpha_n)^2(\nu^*)^{-2}<1$, then  \eqref{est:iv} is valid if
$$
 (1-\alpha_n)^2 \left[ a^2(2x^*+R)^2 + 1 - b^2 (1-\beta_n)^2   \right] +
b^2 (1-\beta_n)^2 \leq  \mu^*<1
$$
for some $\mu^* \in (0,1)$,
which gives the estimate, in the absence of $y$-control ($\beta_n \equiv 0$), that the control intensity required is  
$\alpha_n \geq \alpha_* >0.53$ for small $R$ and $\alpha_* >0.613$ for $R=0.36$. Let $\beta_n \equiv 0.9$,  this  
slightly improves to smaller $\alpha_* \approx 0.51$ and $\alpha_* \approx 0.6$, respectively.

The constants defined in \eqref{def:cal_A_coef} imply \eqref{eq:Lipschitz} for \eqref{henon_toc_diag} with $(x^*,y^*) \approx (0.6314,0.1894)$ and also apply for $a \in (0,1.4)$. Similar estimates can be developed for the other equilibrium, as well as for arbitrary $a,b>0$.
%Thus Corollary~\ref{th:toc_determ_var} 
%yields

\begin{proposition}
\label{prop_henon}
Let $(x^*,y^*)$ be an equilibrium of \eqref{henon}, and $R$ %in the domain $B_R(X^*)$ 
satisfy $R<|x^*|$. Then there exist constants $\alpha_*=\alpha_*(R)$ and $\beta_*=\beta_*(R)$ such that
for any $\alpha_n \in (\alpha_*,1)$, $\beta_n \in (\beta_*,1)$, a solution of \eqref{henon_toc_diag} converges to $(x^*,y^*)$ for any initial value $(x_0,y_0)\in  B_R(X^*)$. 
\end{proposition}
\begin{proof}
For an arbitrary norm and $R<|x^*|$, map \eqref{henon} satisfies 
\eqref{eq:Lipschitz} for $x \in B_R(X^*)$ with the Lipschitz constant $L(R)$.
The constant $L$ is also norm-dependent but the norms are equivalent in ${\mathbb R}^2$, we fix the norm, the radius $R$ and the prescribed convergence rate $\nu^* \in (0,1)$. Denote $\displaystyle \alpha_* = 1 - \frac{\nu^*}{\max\{1,L \| I \| \}} \in (0,1)$, $\beta_* =\alpha_*$.
Then for any $\alpha_n \in (\alpha_*,1)$, $\beta_n \in (\beta_*,1)$, a solution of \eqref{henon_toc_diag} with $(x_0,y_0)\in  B_R(X^*)$ satisfies
$$
\| (I-U_{n}) [F(X) - X^* ]\| \leq  \| I-U_{n} \| \| F(X) - X^* \| \leq \nu^* \| X- X^* \|,
$$
so \eqref{eq:diff_cond_var} holds, and application of Theorem~\ref{toc_lemma_var} concludes the proof.
\end{proof}

%%%%%%%%%%%%%%%%%%%%%%%%%%

\subsubsection{\bf Lozi}
%{Variable Lozi map}
\label{subsubsec:var_Lozi}
 
Computation  similar to Section \ref{subsubsec:var_Henon} leads to
\begin{equation*}
\begin{split}
\left[ \begin{array}{c} 1-a|x|+y-x^* \\ bx-y^* \end{array} \right] & = \left[ \begin{array}{c} (1-ax+y)- (1-ax^*+y^*) \\ b(x-x^*) \end{array} \right] \\ & = 
 \left[ \begin{array}{cc} - a  & 1 \\ b & 0 \end{array} \right] \left[ \begin{array}{c} x - x^* \\ y-y^* \end{array} \right] =: {\bar A} (X-X^*),
\end{split}
\end{equation*}
where $x \in (x^*-R, x^*+R)$, $R<x^*\approx 0.4762$.   Thus  \eqref{eq:Lipschitz} holds for the Lozi map with
\begin{equation}
\label{def:cal_B}
A_{11} =  a,~~ A_{12} = 1, ~~ A_{21} = b,~~ A_{22}=0,~~L_1 = a + 1, ~~L_2=b.
\end{equation}  
The norm of $\displaystyle 
(I-U_n){\mathcal A} = \left[ \begin{array}{cc} a (1-\alpha_n) & 1-\alpha_n \\ 
(1-\beta_n) b & 0 \end{array} \right]
$
should not exceed $\nu^*<1$. 
Further, let us study the dependency of the lower bound on $\alpha_n$ on the choice of the matrix norm.
In ${\mathbf l}_{\infty}$-norm, $\beta_n \in [0,1)$ 
can be arbitrary as $b =0.3 \in (0,1)$.  
For $\alpha_n$, to achieve convergence, $(1-\alpha_n)(a+1)\leq \nu^* <1$ should be satisfied, or for $\alpha_n \geq  \alpha_* > 1 - 1/2.4 \approx 0.584$  stabilization is achieved.

For ${\mathbf l}_{1}$-norm, we should have $(1-\alpha_n) a+ (1-\beta_n) b<1$, the condition 
$1-\alpha_n <1$ is true for any non-zero control level. Thus, for $\beta_n=0$, we get 
$\alpha_n  > 1 - 0.7/1.4=0.5$, so the equilibrium is stable for any $\alpha_n>0.5$. If $\beta_n=0.9$, $\alpha_n > 1-0.97/1.4 \approx 0.31$.

As for the spectral norm, the matrix 
$$
M_n = ((I-U_n){\mathcal   A})'((I-U_n){\mathcal A}) =
\left[ \begin{array}{cc} (1-\alpha_n)^2a^2 + (1-\beta_n)^2 b^2 & (1-\alpha_n)^2 a  \\ (1-\alpha_n)^2 a & (1-\alpha_n)^2 \end{array} \right]
$$
has the trace $(1-\alpha_n)^2a^2 + (1-\beta_n)^2 b^2+ (1-\alpha_n)^2$ and the determinant 
$(1-\beta_n)^2 b^2 (1-\alpha_n)^2$. 
The spectral norm does not exceed $\nu^*<1$ if the larger of two real nonnegative eigenvalues of the quadratic equation $\lambda^2 - (\rm{tr}~M_n) \lambda + \rm{det}~M_n=0 $
is less than $(\nu^*)^2<1$.  Thus $\displaystyle \rm{tr}~M_n + \sqrt{\rm{tr}^2~M_n - 4 \rm{det}~M_n} 
\leq 2 (\nu^*)^2$, which after some simplifications is equivalent to $\rm{tr}~M_n < \rm{det}~M_n/(\nu^*)^2 + (\nu^*)^2$, or
\begin{equation}
\label{auxiliary}
(1-\alpha_n)^2a^2 + (1-\beta_n)^2 b^2+(1-\alpha_n)^2< \frac{(1-\beta_n)^2 b^2 (1-\alpha_n)^2}{(\nu^*)^2} + (\nu^*)^2.
\end{equation}
If $\displaystyle (1-\alpha_n)^2 < \frac{1-(1-\beta_n)^2 b^2}{a^2+1-(1-\beta_n)^2 b^2}$, or \eqref{auxiliary} holds with $\nu^*=1$, from continuity of the right-hand side of \eqref{auxiliary} in $\nu^*$ and strictness of the inequality, there is $\nu^* <1$ such that \eqref{auxiliary} holds, and the spectral norm is not greater than $\nu^* <1$.

For $\beta_n \equiv 0$, we get $\alpha_n \geq \alpha_* >0.44$, while $\beta_n \equiv 0.9$ leads to $\alpha_n\geq \alpha_* >0.42$.

\medskip

The constants defined in \eqref{def:cal_B} imply \eqref{eq:Lipschitz} for \eqref{lozi_toc_diag}, and they are exactly the same in the case 	$x^*<0$ and $x < 0$. 
The proof of the following result is similar to the proof of  
Proposition~\ref{prop_henon}.

%From Corollary~\ref{th:toc_determ_var}, we get the following proposition.

\begin{proposition}
\label{prop_lozi}
Let $(x^*,y^*)$ be a fixed equilibrium of \eqref{lozi}, $R<|x^*|$. Then there exist constants $\alpha_*=\alpha_*(R)$ and $\beta_*=\beta_*(R)$ such that
for any $\alpha_n \in (\alpha_*,1)$, $\beta_n \in (\beta_*,1)$, a solution of \eqref{lozi_toc_diag} converges to $(x^*,y^*)$ for any initial value $(x_0,y_0)\in  B_R(X^*)$. 
\end{proposition}

\begin{remark}
\label{rem:nongom}
Note that the matrix $\mathcal A$ defined in \eqref{def:cal_A_coef}, depends on the radius $R<|x^*|$ in the domain $B_R((x^*,y^*))$, due to $x$ being involved, therefore the bounds for the control parameters guaranteeing  stability can become close to one with the growth of the radius $R$. In contrast to the H\'{e}non case, the matrix $\mathcal A$ in \eqref{def:cal_B} 
for any radius $R<|x^*|$ does not include $x$ explicitly and thus is uniform, see Remark~\ref{rem:absLozi}.
\end{remark}

%%%%%%%%%%%%%%%%%%%%%%%

\section{Stabilization with stochastically perturbed  control}
\label{sec:stoch}

\subsection{General stochastically perturbed  control}
\label{subsec:genstoch}

Stochastic control is a well-developed area, especially for parametric optimization, optimal stochastic control, dynamic programming {\it etc}, see \cite{Astrom} and references therein.
Our approach is based on application of  the  Kolmogorov's Law of Large Numbers, see Lemma \ref{lem:Kolm} below, and is closest to the methods  developed and applied in \cite{BKR2020,BR2017,Medv}  (see also \cite{BR1,BR2019}). Note that we prove only local stability with any a priori fixed probability from $(0,1)$. However, in some cases, computer simulations demonstrate  that for chosen parameters,  stability actually happens  in a relatively wide area of initial parameters, which gives us a hope to extend our local results. However, this extension is left for future research.

We consider a complete filtered probability space $(\Omega, {\mathcal{F}}$, $\{{\mathcal{F}}_n\}_{n \in 
\mathbb N}, {\mathbb P})$, where the filtration $(\mathcal{F}_n)_{n \in \mathbb{N}}$ is naturally generated by 
$m$ sequences of mutually independent identically distributed random variables $\chi_i:=(\chi_{i,n})_{n\in\mathbb{N}}$, $i=1, \dots,m$, i.e. 
$\mathcal{F}_{n} = \sigma \left\{\chi_{i,j},\, j=1, 2, \dots, n, \, i=1, 2, \dots, m \right\}$. The standard abbreviation ``a.s.'' is used for either ``almost sure" or ``almost surely" with respect to a fixed probability measure $\mathbb P$.
%, and ``i.i.d.'' for  ``independent identically distributed'', to describe random variables. 
For a detailed introduction of stochastic concepts and notations we refer the reader to \cite{Shiryaev96}. 

In this paper  we reduce our investigation only to the case when control parameters in VMTOC model~\eqref{eq:VMTOC_var} are perturbed by a bounded noise, since in many real-world models, in particular, in population dynamics, noise amplitudes are bounded. 
\begin{assumption}
\label{as:noise}
$\chi_i:=(\chi_{i,n})_{n\in\mathbb{N}}$, $i=1, \dots,m$ are $m$ sequences of mutually independent identically distributed random variables such that $|\chi_{i,n}|\le 1$, $i=1, \dots, m$, $n\in\mathbb{N}$.
\end{assumption}
Note that, in general, random variables $\chi_i$  can have different distributions.

We study VMTOC \eqref{eq:VMTOC_var},  with an infinite family of  diagonal $m \times m$ matrices $\{ U_n \}$, \\ $U_n=$diag$(d_{1,n}, \dots, d_{m,n})$, where controls $d_{i,n}$ are stochastically perturbed
\begin{equation}
\label{def:d}
d_{i,n}=\alpha_i+\ell_i\chi_{i,n}, \quad \alpha_i\in (0, 1), \quad \ell_i\in [0, \min\{\alpha_i, 1-\alpha_i\} ), \quad i=1, \dots, m.
\end{equation}
Here $\chi_{i,n}$ are random variables satisfying Assumption \ref{as:noise}, 
 $\ell_i $  are the intensities of noises.
Note that \eqref{def:d} guarantees that $d_{i,n}=\alpha_i+\ell_i\chi_{i,n}\in (0, 1)$, $i=1, \dots, m$, $n\in\mathbb{N}$.

Instead of condition \eqref{eq:diff_cond_var} we assume now that  
\begin{equation}
\label{eq:diff_cond_stoch}
\| (I-U_n) [F(X) - X^* ]\| \leq \nu(n) \|X - X^*\|,\quad n\in \mathbb N , \quad X\in  B_R(X^*),
\end{equation}
where
\begin{equation}
\label{def:phi}
\nu(n)=\phi(d_{1,n}, \dots, d_{m,n}), \quad \phi:\mathbb R^m\to [0, \infty) \quad \mbox{is a continuous function}.
\end{equation}
Since  the function $\phi$ from \eqref{def:phi} is continuous, random variables $\chi_{i,n}$, $i=1, \dots, m$, $n\in \mathbb N$, are mutually independent, and $\chi_{i,n}$ are identically distributed (for each $i$),  random variables $\nu(n)$ are also mutually independent and identically  distributed. 

For each $\omega\in \Omega$, the stochastic controls $\alpha_i + \ell_i \chi_{i, n+1}(\omega)$, $i=1, \dots, m$, $n\in \mathbb N_0$, are just numbers in the corresponding intervals $[\alpha_i - \ell_i, \alpha_i + \ell_i]$. Therefore, if we can choose $\alpha_i$ and $\ell_i$ in such a way that $\nu(n)\le \nu^*<1$ on $\Omega$, for some $\nu^*\in (0, 1)$, we are in the same situation as in Section \ref{sec_main}, see Propositions \ref{prop_henon_stoch} and \ref{prop_lozi_stoch} in Section \ref{subsec:stochHL}. The idea to some extent generalizes the approach in \cite{BKR} to systems. 

In this section we concentrate on the case when $\nu(n)$ might  exceed one on some set with a nonzero probability, assuming instead  that
\begin{equation}
\label{cond:mainstoch}
\mathbb E \ln  \nu(n)=-\lambda_1<0
\end{equation}
and applying the  Kolmogorov's Law of Large Numbers, see \cite[Page 391]{Shiryaev96}. 

\begin{lemma}
\label{lem:Kolm}[Kolmogorov's Law of Large Numbers]
Let $(v_{n})_{n\in\ \mathbb N}$ be a sequence of independent identically distributed random variables, where $\mathbb E |v_n|<\infty$, $n\in\mathbb{N}$, their common mean is $\mu:=\mathbb E v_n$, and the partial sum is $\displaystyle S_n:= \sum_{k=1}^n v_k$.	Then $
\displaystyle \lim_{n\to\infty} \frac{S_n}{n} = \mu$, a.s.
\end{lemma}

Even though the result of Theorem \ref{thm:KLLNstoch} is new,  its proof  is quite standard, and we give  it here only for completeness of presentation. Since random variables $\chi_{i,n}$, $i=1, \dots, m$, $n\in \mathbb N$, are bounded, we follow the approach recently employed in \cite{BKR2020}.

\begin{theorem}
\label{thm:KLLNstoch}
Let Assumption \ref{as:noise} and conditions \eqref{eq:diff_cond_stoch}, \eqref{def:phi}, \eqref {cond:mainstoch} hold. Then,  for each $\gamma\in (0, 1)$, there is $\Omega(\gamma)\subseteq \Omega$, $\mathbb P(\Omega(\gamma))>1-\gamma$, and $0<\delta<R$ such that for the initial values $X_0\in B_R(X^*)$ satisfying $\|X_0-X^*\|\le \delta$, a solution of \eqref{eq:VMTOC_var} converges to $X^*$ on  $\Omega(\gamma)$.
\end{theorem}

\begin{proof}
Condition \eqref{eq:diff_cond_stoch} implies 
\begin{equation*}
\| (I-U_n) [F(X) - X^* ]\| \leq \exp\{\ln \nu(n)\} \|X - X^*\|,\quad n\in \mathbb N , \quad X\in B_R(X^*).
\end{equation*}
Since  $\left(\ln  \nu(n) \right)_{n\in \mathbb N}$ 
is a sequence of independent  identically distributed bounded random variables, Lemma \ref{lem:Kolm}  (Kolmogorov's Law of Large Numbers) can be applied. So, by \eqref{cond:mainstoch}, there is 
 a random integer $\mathcal N(\omega)$ such that,  for  $n\ge \mathcal N(\omega)$, a.s.,
\begin{equation}
\label{rel:1}
\sum_{i=1}^n \ln \mathcal \nu(i)<-\frac{\lambda_1}2 n.% \quad n\ge \mathcal N(\varepsilon, \omega).
\end{equation}
Therefore, for each $\gamma\in (0, 1)$, there exists a nonrandom integer $N=N(\gamma)$ and a set  $\Omega(\gamma)\subseteq \Omega$ with $\mathbb P(\Omega(\gamma))>1-\gamma$ such that  \eqref{rel:1} holds for $n\ge N(\gamma)$ on $\Omega(\gamma)$.

Since $d_{i,n}\in [0, 1)$,  $i=1, \dots, m$, $n\in \mathbb N$, see \eqref{def:d}, and by continuity of $\phi$, see \eqref{def:phi}, there exists a nonrandom number $\mathcal M>1$ such that, $\forall n\in \mathbb N$, on $\Omega$,
\[
\nu(n) <\mathcal M.
\]
Choose now $\delta=R\mathcal M^{-N(\gamma)}$ and let $\|X_0-X^*\|<\delta$. Then
\begin{equation*}
%\label{calc1}
\begin{split}
\|X_{1} - X^*\| & =\| (I-U_1) [F(X_0) - X^* ]\| \leq \nu(1)\|X_0 - X^*\|<\mathcal M \delta<R, \\ & \mbox{so} \quad X_1\in B_R(X^*),%\mathcal D,
\\
\|X_{2} - X^*\| & =\| (I-U_2) [F(X_1) - X^* ]\| \leq \nu(2)\|X_1 - X^*\|<\nu(2)\nu(1)\|X_0 - X^*\|\\
&=\exp\left\{\sum_{i=1}^2\ln \mathcal \nu(i)\right\} \|X_0 - X^*\|<
\mathcal M^2 \delta<R, \quad \mbox{so} \quad X_2\in B_R(X^*),%\mathcal D,
\\
&\cdots\\
\|X_{N} - X^*\| & =\| (I-U_N) [F(X_{N-1}) - X^* ]\| \leq \nu(N)\|X_{N-2} - X^*\|\\&
 \leq \exp\left\{\sum_{i=1}^N \ln \mathcal \nu(i)\right\}\|X_0 - X^*\|<\mathcal M^N \delta<R, \quad \mbox{so} \quad X_N\in B_R(X^*).%\mathcal D.
\end{split}
\end{equation*}
Then we can keep solution $X_n$ in $B_R(X^*)$ for at least $N=N(\gamma)$ steps, until estimate \eqref{rel:1} (the Kolmogorov's Law of Large Numbers) starts working on $\Omega(\gamma)$. Further, we get, on $\Omega(\gamma)$,
\begin{equation*}
%\label{calc2}
\begin{split}
\|X_{N+1} - X^*\| & =\| (I-U_{N+1}) [F(X_N) - X^* ]\| \leq \nu(N+1)\|X_N - X^*\|\\&
 \leq \exp\left\{\sum_{i=1}^{N+1} \ln \mathcal \nu(i)\right\}\|X_0 - X^*\|
 \leq \exp\left\{-\frac{\lambda_1(N+1)}2\right\}\|X_0 - X^*\| \\ & \leq \|X_0 - X^*\|<\delta,
\quad \mbox{and for any} \quad n>N+1,\\
\|X_{n} - X^*\| & \leq \exp\left\{\sum_{i=0}^{n} \ln \mathcal \nu(i)\right\}\|X_0 - X^*\|
\leq \exp\left\{-\frac{\lambda_1n}2\right\} \|X_0 - X^*\|\\
& < \exp\left\{-\frac{\lambda_1n}2\right\}\delta\to 0, \quad \mbox{as} \quad n\to \infty,
\end{split}
\end{equation*}
which provides convergence of a solution to $X^*$ on $\Omega(\gamma)$. 
\end{proof}  

%%%%%%%%%%%%%%%
\subsection{H\'{e}non and Lozi maps with a stochastic control}
\label {subsec:stochHL} 

Let $d_{1,n}$ and $d_{2,n}$ be defined as in \eqref{def:d}.  
Following the methods and the notations of Section~\ref{subsec:varHL},  we conclude that stochastic VMTOC of  H\'{e}non \eqref{henon_toc_diag} and  Lozi  \eqref{lozi_toc_diag} maps  can be transformed to the equation
\begin{equation*}
%\label{eq:TOCstoch}
(I-U_{n+1})[F(X_n)-X^*]
=\mathcal C_{n+1}(x_n) \left[ \begin{array}{c} x_n - x^* \\ y_n-y^* \end{array} \right].
\end{equation*}
In the case of the H\'{e}non map \eqref{henon_toc_diag} we have
\begin{equation*}
%\label{def:CHen}
\mathcal C_{n+1}(x_n)=:\left[ \begin{array}{cc} -a\left( 1-\alpha_1 - \ell_1 \chi_{1, n+1} \right) 
(x_n+ x^*) & \left(1- \alpha_1 - \ell_1 \chi_{1, n+1} \right) \\ \left( 1-\alpha_2 - \ell_2 \chi_{2, n+1} \right)  b & 0 \end{array} \right],
\end{equation*}
while, in the case of the Lozi map \eqref{lozi_toc_diag} and $x_nx^*>0$, 
\begin{equation*}
%\label{def:CLozi}
%\mathcal C_{n}(x)=
\mathcal C_{n+1}(x_n)=\mathcal C_{n+1}=:\left[ \begin{array}{cc} -a\left( 1-\alpha_1 - \ell_1 \chi_{1, n+1} \right) 
& \left(1- \alpha_1 - \ell_1 \chi_{1, n+1} \right) \\ b\left( 1-\alpha_2 - \ell_2 \chi_{2, n+1} \right)& 0 \end{array} \right].
\end{equation*}
In  general,  the norms $\|\mathcal C_n\|$ are  not less than one on all  $\Omega$, and we get stability only on  a smaller set $\Omega_\gamma\subset \Omega$. If however, the coefficients $\alpha_i$ and $\ell_i$, $i=1,2$, 
are chosen in such a way that $\|\mathcal C_n\|\le \mu^*<1$ on $\Omega$, for all $n\in \mathbb N$, we are in the same situation as in Section~\ref{sec_main}.  For the H\'{e}non map it holds  when $\alpha_*<\alpha_1-\ell_1$ and $\beta_*<\alpha_2-\ell_2$,  since $\alpha_1-\ell_1\le \alpha_1+\ell_1\chi_{1,n} $ and $\alpha_2-\ell_2\le \alpha_1+\ell_1 \chi_{2, n} $, for each $n\in \mathbb N$.
Here $\alpha_*$ and $\beta_*$ are the convergence bounds from Proposition~\ref{prop_henon}. The similar arguments are applied for the Lozi map.
Therefore, Propositions~\ref{prop_henon} and \ref{prop_lozi}  immediately imply the following results.

\begin{proposition}
\label{prop_henon_stoch}
Let $(x^*,y^*)$ be an equilibrium of \eqref{henon}, $R<|x^*|$, and constants $\alpha_*$ and $\beta_*$ be the convergence bounds from Proposition~\ref{prop_henon}. Let $X$ be a solution  to  \eqref{henon_toc_diag}
with $d_{1,n}$ and $d_{2,n}$ defined as in \eqref{def:d}. 
Then for any $\alpha_1 \in (\alpha_*,1)$, $\alpha_2 \in (\beta_*,1)$, $\ell_1<\min\{ \alpha_1 - \alpha_*, 1-\alpha_1\}$ and $\ell_2<\min\{ \alpha_2 - \beta_*, 1-\alpha_2\}$, $X$ converges to $(x^*,y^*)$ for any initial value $(x_0,y_0)\in  B_R(X^*)$. 
\end{proposition}
\begin{proposition}
\label{prop_lozi_stoch}
Let $(x^*,y^*)$ be an equilibrium of \eqref{lozi}, $R<|x^*|$, and constants $\alpha_*$ and $\beta_*$ be the convergence bounds from Proposition~\ref{prop_lozi}. Let $X$ be a solution  to  \eqref{lozi_toc_diag}
with $d_{1,n}$ and $d_{2,n}$ defined as in \eqref{def:d}. Then for any 
$\alpha_1 \in (\alpha_*,1)$, $\alpha_2 \in (\beta_*,1)$, $\ell_1<\min\{ \alpha_1 - \alpha_*, 1-\alpha_1\}$ and $\ell_2<\min\{ \alpha_2 - \beta_*, 1-\alpha_2\}$, $X$ converges to $(x^*,y^*)$ for any initial value $(x_0,y_0)\in  B_R(X^*)$. 
\end{proposition}
Proceed now to a more general situation, when introduction of noise contributes to reducing the minimum values of control parameters which guarantee stability. 

In order to apply Theorem~\ref{thm:KLLNstoch}, we need to find 
control parameters $\alpha_1, \alpha_2, \ell_1, \ell_2$ and $\lambda_1>0$ such that  condition \eqref{cond:mainstoch} holds for some $\nu(n)>\max_{X\in B(X^*, R)}\| \mathcal C_{n}(x)\|$. Note that $\nu(n)$ and values of  parameters depend on the particular norm $\| \cdot\|$ and distributions of noises $\chi_1$ and $\chi_2$. 
In the following Sections \ref{subsec:stochHen} and \ref{subsec:stochLozi}, we consider $\|\cdot \|_\infty$ and $\|\cdot \|_1$, Bernoulli and uniform continuous distributions for the noise, 
and illustrate obtained results with computer simulations.

\subsubsection{\bf H\'{e}non}
\label{subsec:stochHen}

\begin{enumerate}
\item {\it Norm $\|\cdot \|_\infty$, $\chi_1$ is Bernoulli distributed.}
Since $|x+x^*|\le 2x^*+R$, we have, for $X\in B_R(X^*)$,
\begin{equation*}
\| \mathcal C_{n}(x)\|_\infty\le\max\left\{\left(1- \alpha_1 - \ell_1 \chi_{1, n} \right)[a(2x^*+R)+1], \, \left( 1-\alpha_2 - \ell_2 \chi_{2, n} \right)  b \right\}.
%\\<\max\left\{\left( 1-\alpha_1 - \ell_1 \chi_{1, n} \right)[1.4(1.2626+R)+1], \, 0.3\right\}.
\end{equation*}
Theorem \ref{thm:KLLNstoch} states only local stability for any given probability, therefore we can assume 
$R \approx 7.1\times 10^{-5}$, so that $a(2x^*+R) \approx 1.7677$ and use local estimates at the equilibrium point for 
$\| \mathcal C_{n}(x)\|_\infty$:
\begin{equation*}
\| \mathcal C_{n}(x)\|_\infty\le\max\left\{2.7677 \left(1- \alpha_1 - \ell_1 \chi_{1, n} \right), \, 0.3\left( 1-\alpha_2 - \ell_2 \chi_{2, n} \right)  \right\}.
\end{equation*}
If 
\begin{equation}
\label{ineq:a12}
  1-\alpha_2 + \ell_2<9.2256 \left(1- \alpha_1 - \ell_1 \right)
 \end{equation}
then
$
\max\left\{2.7677 \left(1- \alpha_1 - \ell_1 \chi_{1, n} \right), \, 0.3\left( 1-\alpha_2 - \ell_2 \chi_{2, n} \right)  \right\}=2.7677 \left(1- \alpha_1 - \ell_1 \chi_{1, n} \right).
$
%if $2.7676\left(1- \alpha_1 - \ell_1 \right)>0.3( 1-\alpha_2 + \ell_2)$, or 
%\begin{equation}
%\label{ineq:a12}
%  1-\alpha_2 + \ell_2<9.2253\left(1- \alpha_1 - \ell_1 \right).
% \end{equation}

Choose first $\alpha_1, \ell_1$ such that  condition \eqref{cond:mainstoch} holds for 
 $\nu(n) := 2.7677 \left(1- \alpha_1 - \ell_1 \chi_{1, n} \right)$ and for some $\lambda_1>0$.
Since $\chi_1$ is Bernoulli distributed, we get $\mathbb E \ln \nu(n)=\ln  2.7677 +0.5\ln \left( (1-\alpha_1)^2- \ell_1^2\right)$, which leads to the following estimation 
\[
\ell_1^2>(1-\alpha_1)^2-0.1305.
\]
%We have 
%\begin{multline*}
%%\| \mathcal C_{n}(x)\|_1\le 2.7676\left( 1-\alpha_1 -\ell_1 \chi_{1, n} \right),\quad 
%\mathbb E \ln \nu(n)%=\mathbb E \left[\ln  3.006+0.5\left(\ln \left( 1-\alpha_1 + \ell_1 \right) +\ln \left( 1-\alpha_1 - \ell_1\right) \right)
%%\right]
%=\ln  2.7676+0.5\ln \left( (1-\alpha_1)^2- \ell_1^2\right),\\
%0>\ln  2.7676+0.5\ln \left( (1-\alpha_1)^2- \ell_1^2\right)=\ln  \left[2.7676\sqrt{  (1-\alpha_1)^2- \ell_1^2}\right],\\
%2.7676\sqrt{  (1-\alpha_1)^2- \ell_1^2}<1, \quad (1-\alpha_1)^2- \ell_1^2<0.1305, \quad \ell_1^2>(1-\alpha_1)^2-0.1305.
%\end{multline*}
To construct the interval for $\ell_1$ in \eqref{def:d}, we take $\alpha_1^2>(1-\alpha_1)^2-0.1305$,
which gives $\alpha_1>\frac{0.8695}2=0.4347$, so, for $\alpha_1=0.44$ we have $\ell_1>\sqrt{0.1831}=0.4279$.  The right-hand side of the inequality in condition \eqref {ineq:a12} for $\alpha_1=0.44$, $\ell_1=0.4279$ gives $1.2186 > 1$, which means that  \eqref {ineq:a12} holds for each $\alpha_2, \ell_2$ satisfying \eqref{def:d}.

For $\alpha_1=0.44$ and no control in $y$ ($\alpha_2 = \ell_2 = 0$),  
Fig.~\ref{figure_ex_1_fig_3} confirms that there is stabilization already for $\ell_1 = 0.3$.
The series of runs with the same initial conditions illustrate that, without noise in the control, we have a stable two-cycle (Fig.~\ref{figure_ex_1_fig_3} (a)). As small noise with $\ell_1=0.15$ is introduced, a noisy around this two-cycles solution nearly reaches the equilibrium (Fig.~\ref{figure_ex_1_fig_3} (b)). As noise increases to $\ell_1 = 0.25$, the $x$-trajectory already looks as a stochastic perturbation of the equilibrium (Fig.~\ref{figure_ex_1_fig_3} (c)), and the equilibrium becomes stable for $\ell_1=0.3< 0.4279$  (Fig.~\ref{figure_ex_1_fig_3} (d)). We can also observe that much smaller noise does not guarantee stabilization.

%in-text figure
\begin{figure}
%\centering
\subfloat[]{%
\resizebox*{3cm}{!}{\includegraphics{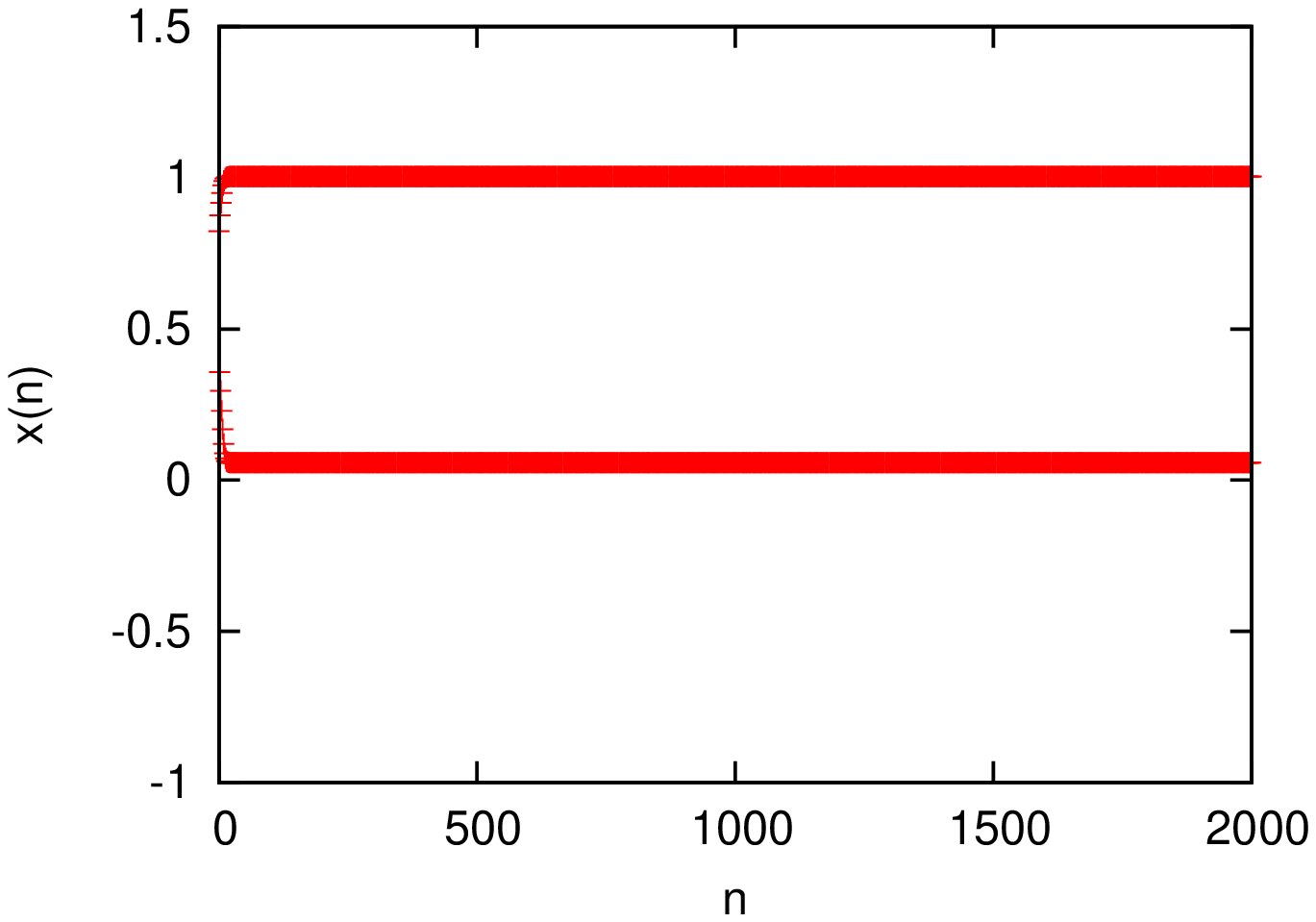}}}
\hspace{5pt}
~~~
\subfloat[]{%
\resizebox*{3cm}{!}{\includegraphics{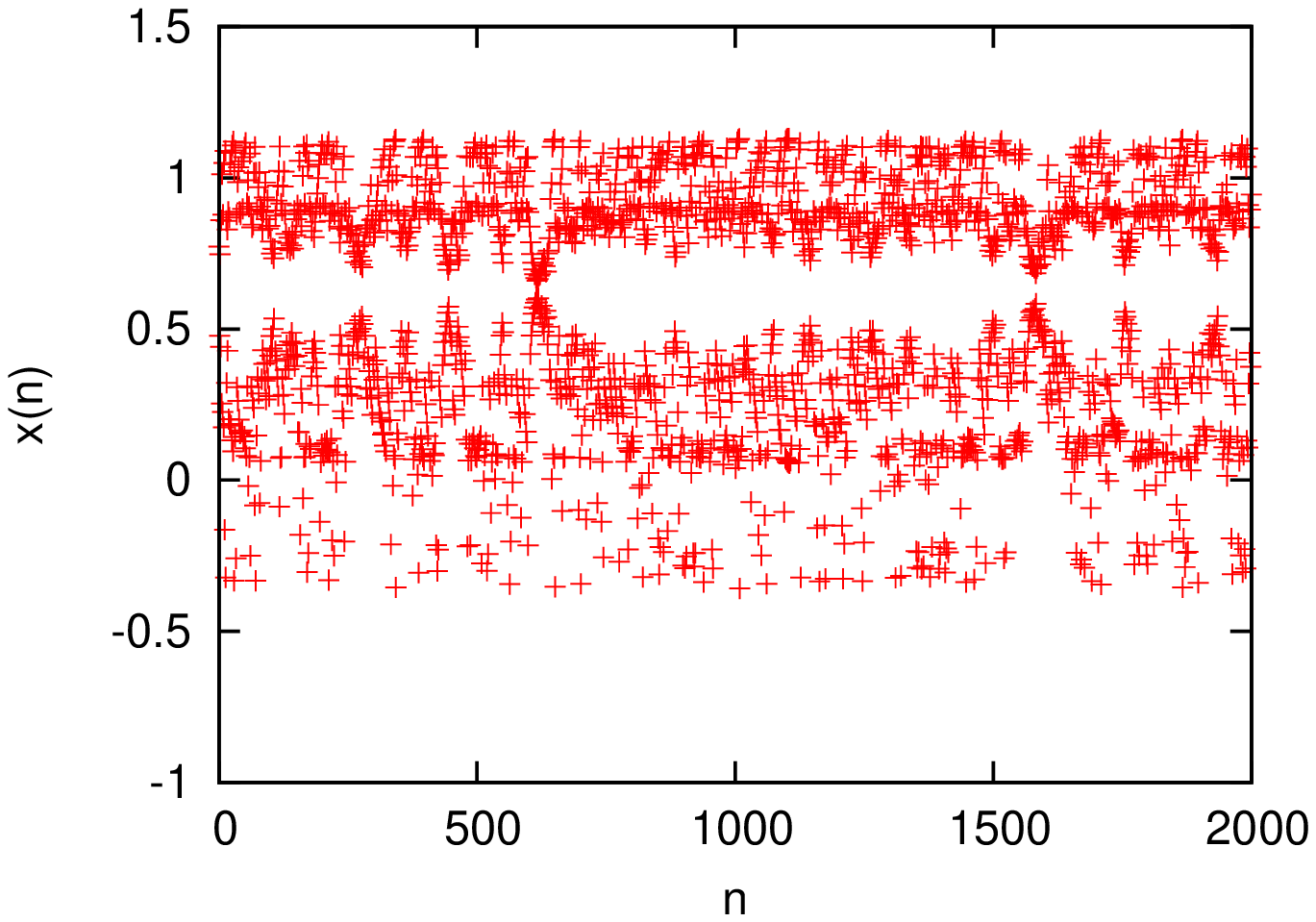}}}
~~~
\subfloat[]{
\resizebox*{3cm}{!}{\includegraphics{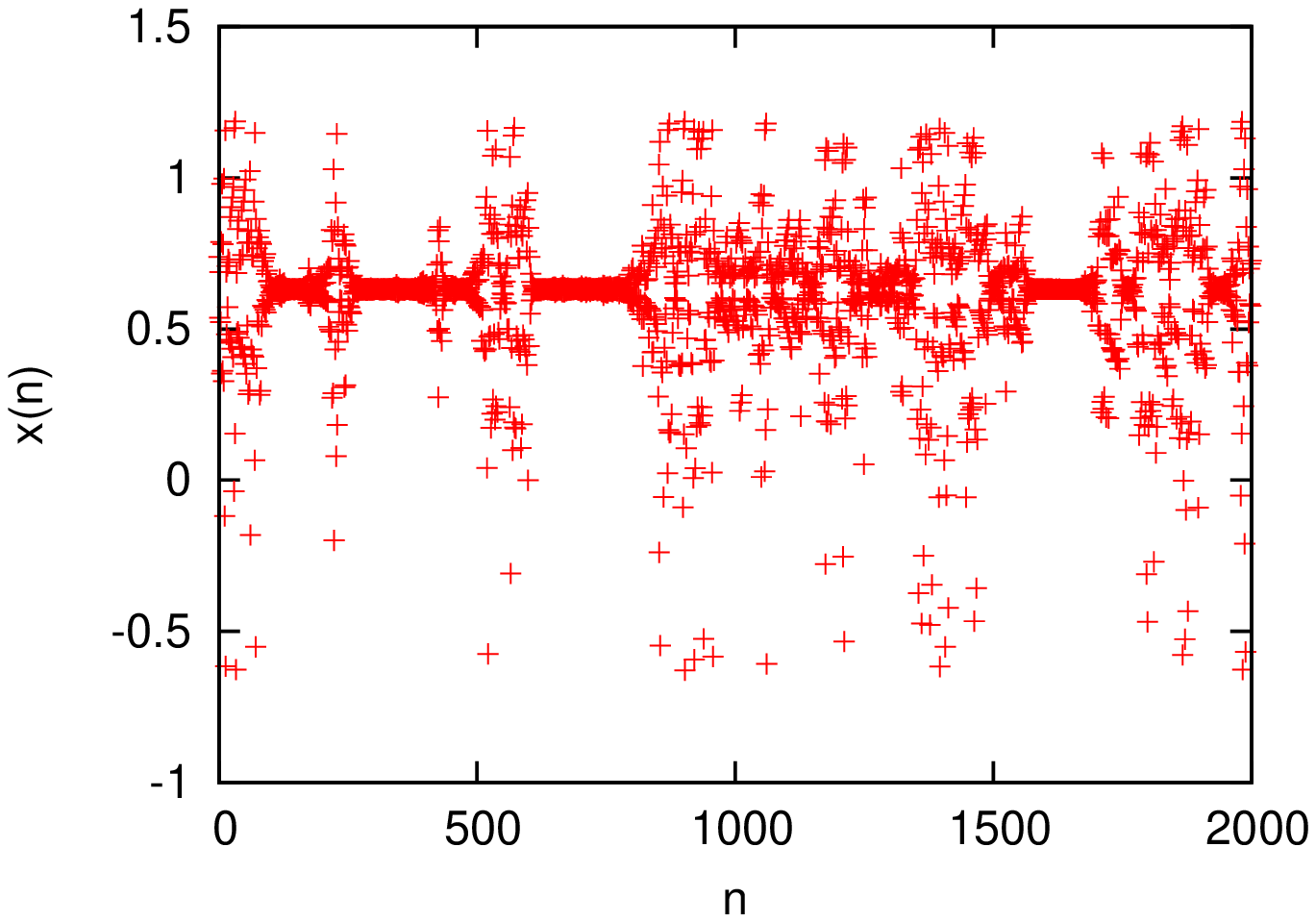}}}
\hspace{5pt}
~~~
\subfloat[]{%
\resizebox*{3cm}{!}{\includegraphics{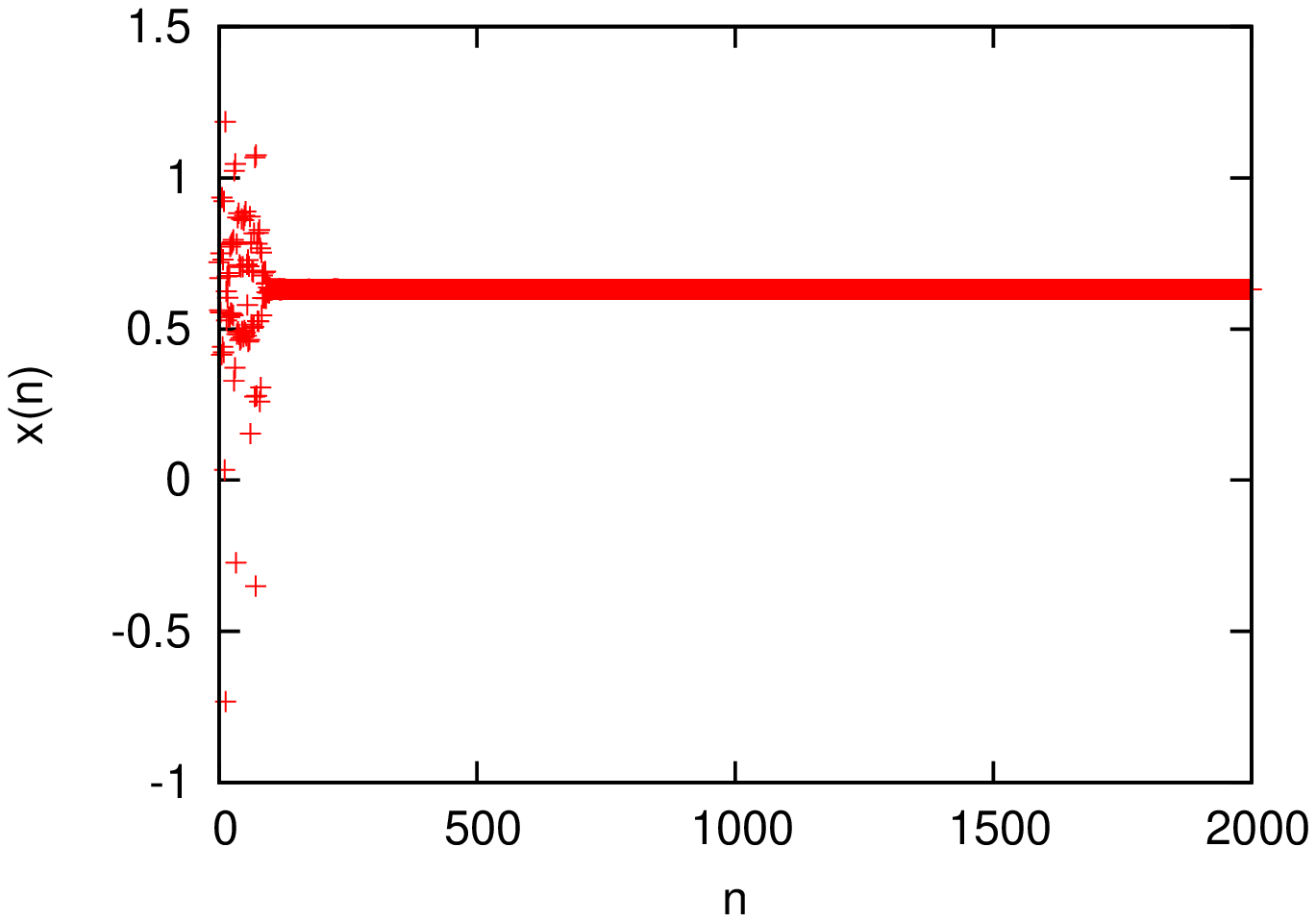}}}
\caption{Runs of the H\'{e}non map for $x$-coordinate only, $\alpha = 0.44$, no control in 
$y$, $x_0=0.3$, $y_0=0.1$ and (a) $\ell = 0$, (b) $\ell = 0.15$, (c) $\ell = 0.25$, (d) $\ell = 0.3$ 
for the Bernoulli noise.
\bigskip
}
\label{figure_ex_1_fig_3}
\end{figure}

Two-dimensional Fig.~\ref{figure_ex_1_fig_4} with the attracting limit sets also illustrates  the way to stabilization for $\alpha=0.3< 0.4279$.
There is a stable two cycle without noise (not illustrated), and close to this cycle blurred set for $\ell_1 = 0.05$ (Fig.~\ref{figure_ex_1_fig_4} (a)), which widens with the increase of the noise to $\ell_1 = 0.15$, still not reaching the equilibrium point (Fig.~\ref{figure_ex_1_fig_4} (b)), getting a blurred equilibrium for $\ell_1=0.25$ (Fig.~\ref{figure_ex_1_fig_4} (c)) and stabilization of the equilibrium point for $\ell_1=0.3$ (Fig.~\ref{figure_ex_1_fig_4} (d)).

%in-text figure
\begin{figure}
%\centering
\subfloat[]{%
\resizebox*{3cm}{!}{\includegraphics{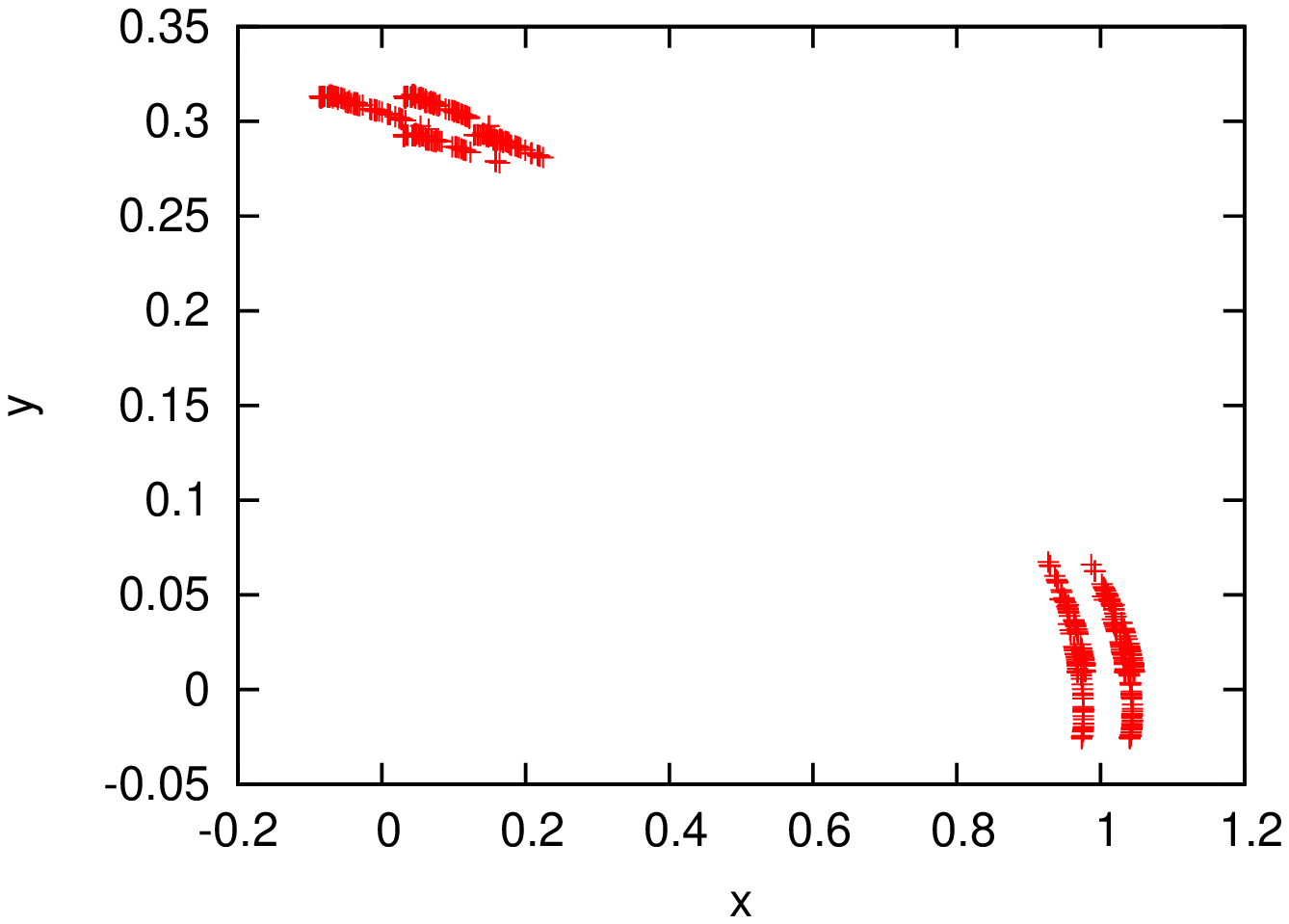}}}
\hspace{5pt}
~~~
\subfloat[]{%
\resizebox*{3cm}{!}{\includegraphics{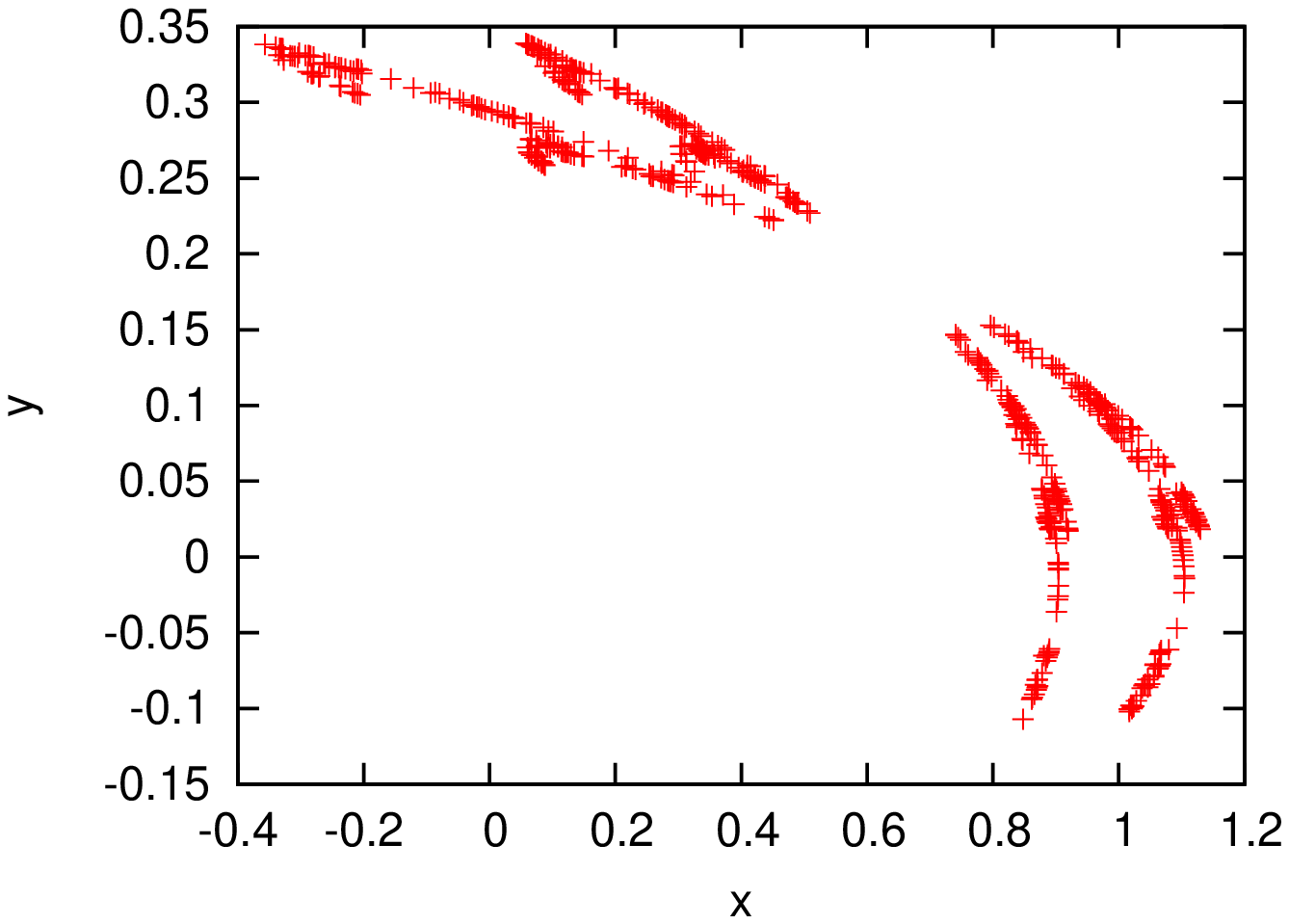}}}
~~~
\subfloat[]{
\resizebox*{3cm}{!}{\includegraphics{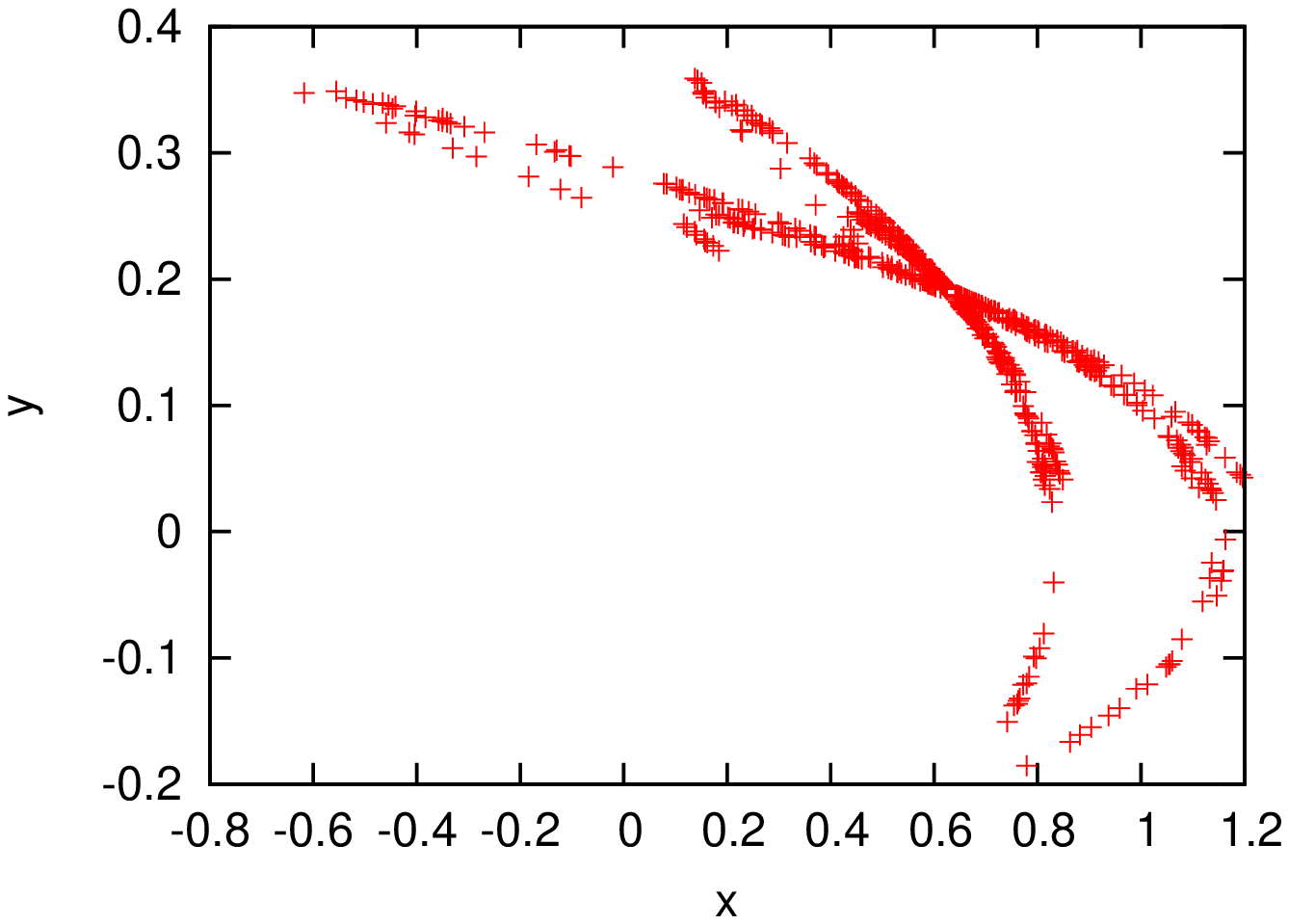}}}
\hspace{5pt}
~~~
\subfloat[]{%
\resizebox*{3cm}{!}{\includegraphics{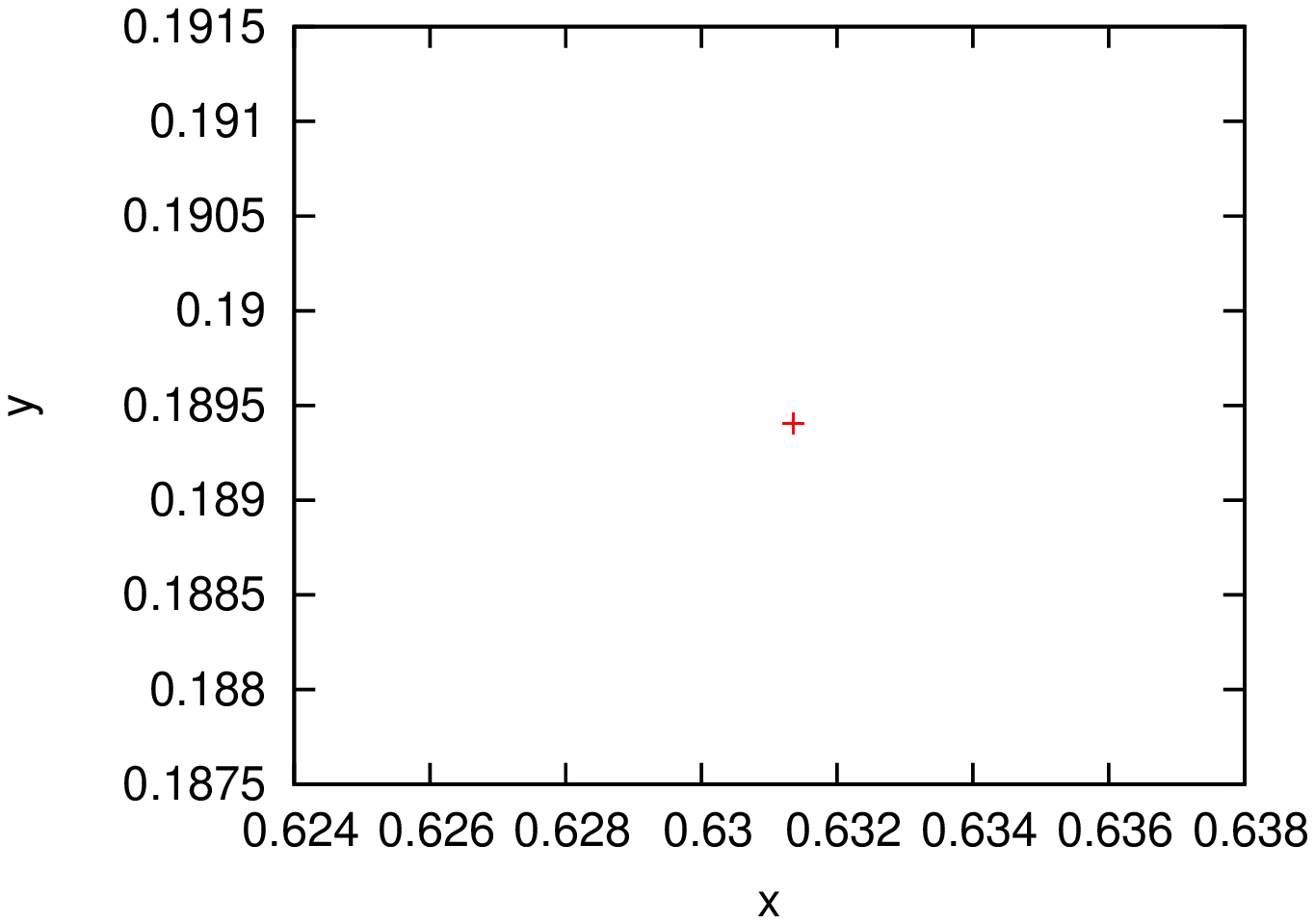}}}
\caption{The limit set for the H\'{e}non map for  $\alpha = 0.44$, no control in
$y$, $x_0 = 0.3$, $y_0=0.1$ and (a) $\ell = 0.05$, (b) $\ell = 0.15$, (c) $\ell = 0.25$, (d) $\ell = 0.3$ for the Bernoulli 
noise.
\bigskip
}
\label{figure_ex_1_fig_4}
\end{figure}

\item 
{\it Norm $\| \cdot\|_1$, $\chi_1$ is Bernoulli distributed, $\ell_2=0$. } Since $2x^*a=1.7676>1$, we have 
\begin{equation*}
\begin{split}
\| \mathcal C_{n}(x)\|_1 \le & \max\left\{\left(1- \alpha_1 - \ell_1 \chi_{1, n} \right)a(2x^*+R)+\left( 1-\alpha_2 - \ell_2 \chi_{2, n} \right)  b, \right. \\ & \left. \left(1- \alpha_1 - \ell_1 \chi_{1, n} \right) \right\}\\
& \approx 1.7677 \left( 1-\alpha_1 - \ell_1 \chi_{1, n} \right)+0.3\left( 1-\alpha_2 - \ell_2 \chi_{2, n} \right)=:\nu(n),\\
\mathbb E \ln \nu(n)& \approx 0.5\ln \left[ \left(1.7677 ( 1-\alpha_1)+0.3( 1-\alpha_2)\right)^2-\left(1.7677 \ell_1\right)^2\right].
\end{split}
\end{equation*}
So $\ell^2_1>\left(( 1-\alpha_1)+0.1697( 1-\alpha_2)\right)^2-0.3200 \implies \mathbb E \ln \nu(n)<0$.
For $\alpha_2=0.8$, $\ell_1=0.2862$ and $\alpha_1=0.4$  we get $\mathbb E \ln \nu(n)=\ln(0.9999)<0$ along with 
the fact that condition \eqref{def:d} holds.

\item{\it Norm $\| \cdot \|_1$, $\chi_1$ has  uniform continuous distribution, $\ell_2=0$. }
We have
\begin{equation*}
\begin{split}
&\mathbb E \ln \nu(n) \approx \frac 12\int_{-1}^1\ln \left[ \left(1.7677 ( 1-\alpha_1)+0.3( 1-\alpha_2)\right)-1.7677 \ell_1z\right]dz.\\
\end{split}
\end{equation*}
The set of parameters 
$\alpha_2=0.9$, $\alpha_1=0.44$, $\ell_1=0.2862$  gives  $\mathbb E \ln \nu(n) \approx -0.0251<0$ and, also, condition \eqref{def:d} holds.

Fig.~\ref{figure_ex_2_fig_1} illustrates and compares introduction of Bernoulli and uniformly distributed on $[-1,1]$ noises for the H\'{e}non map, using a part of the bifurcation diagram for $x$ in the range of parameters of $\alpha_1$ close to the last bifurcation leading to stabilization of the equilibrium, with $\alpha_2=0.8$-control in $y$. In Fig.~\ref{figure_ex_2_fig_1} (a), for $\ell_1=0$, we see a period-halving bifurcation; in (b) and (c) figures (Bernoulli and uniform noise, respectively, both with the amplitude $\ell_1=0.2861$), the value of $\alpha_1$ for this bifurcation is smaller ($\approx 0.39$ and $\approx 0.432$, respectively), with some stabilization advantage of the Bernoulli distribution. 
Fig.~\ref{figure_ex_2_fig_1} also shows that the image for the uniformly distributed perturbation looks `noisier' ((c) compared to (b)). Note that compared to the above computations, Fig.~\ref{figure_ex_2_fig_1} (c) illustrates stabilization for $\ell_1=0.2861$, $\alpha_1 \approx 0.432<0.44$ and
$\alpha_2=0.8<0.9$. The case $\alpha_2=0.9$ is further tested in Fig.~\ref{figure_ex_2_fig_2}.

%in-text figure
\begin{figure}
%\centering
\subfloat[]{%
\resizebox*{3.5cm}{!}{\includegraphics{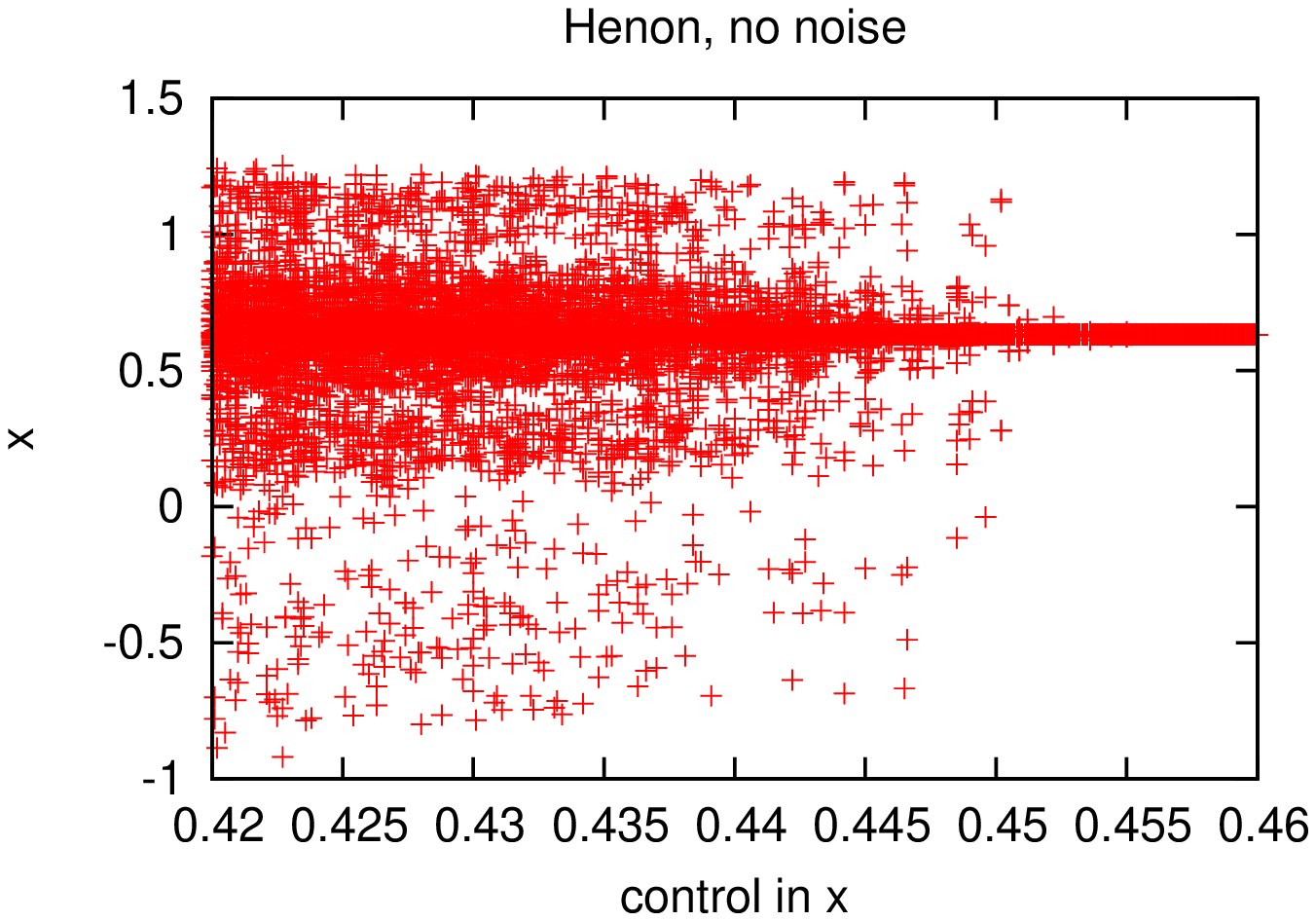}}}
\hspace{12pt}
~~~
\subfloat[]{%
\resizebox*{3.5cm}{!}{\includegraphics{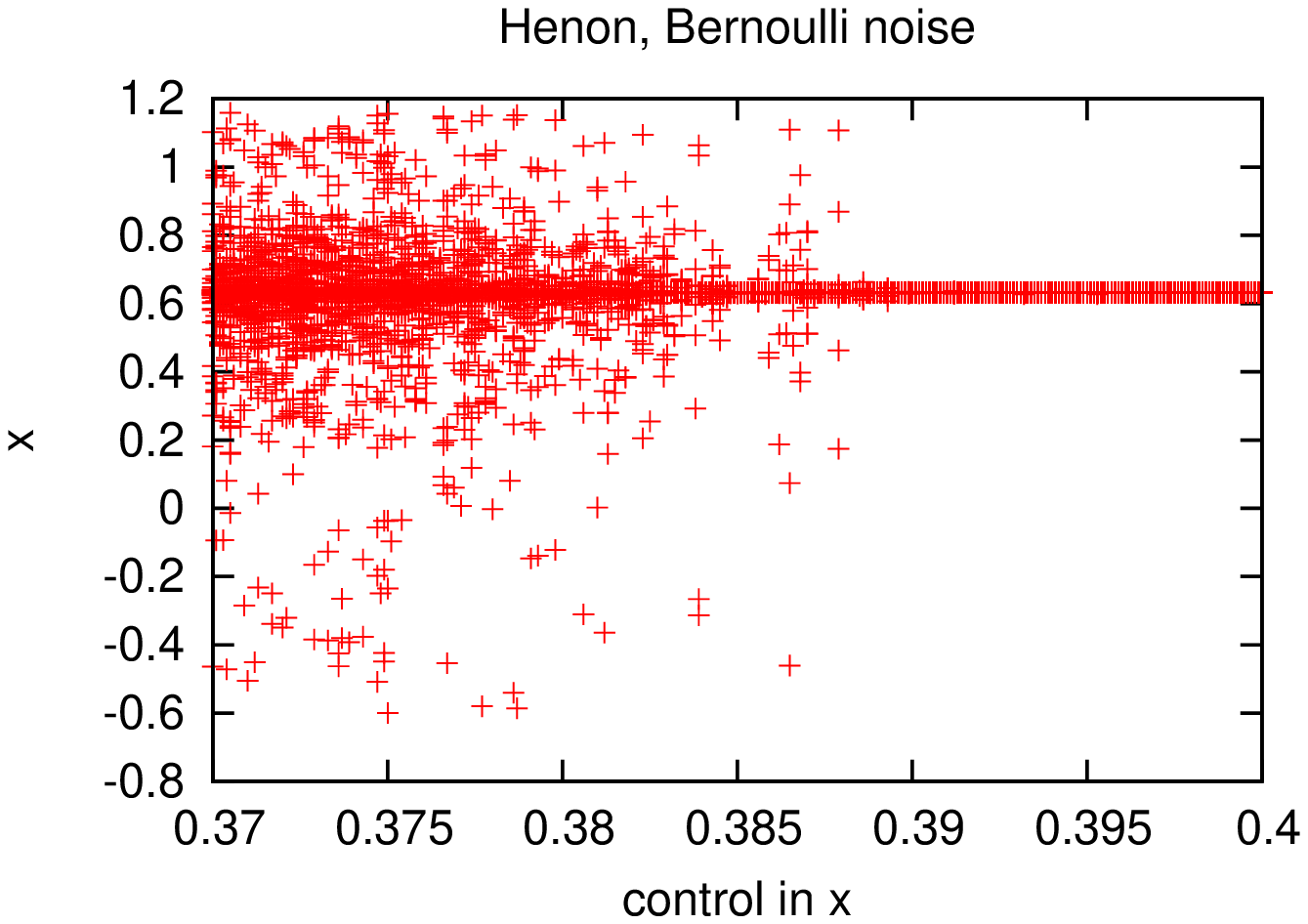}}}
\hspace{12pt}
~~~
\subfloat[]{
\resizebox*{3.5cm}{!}{\includegraphics{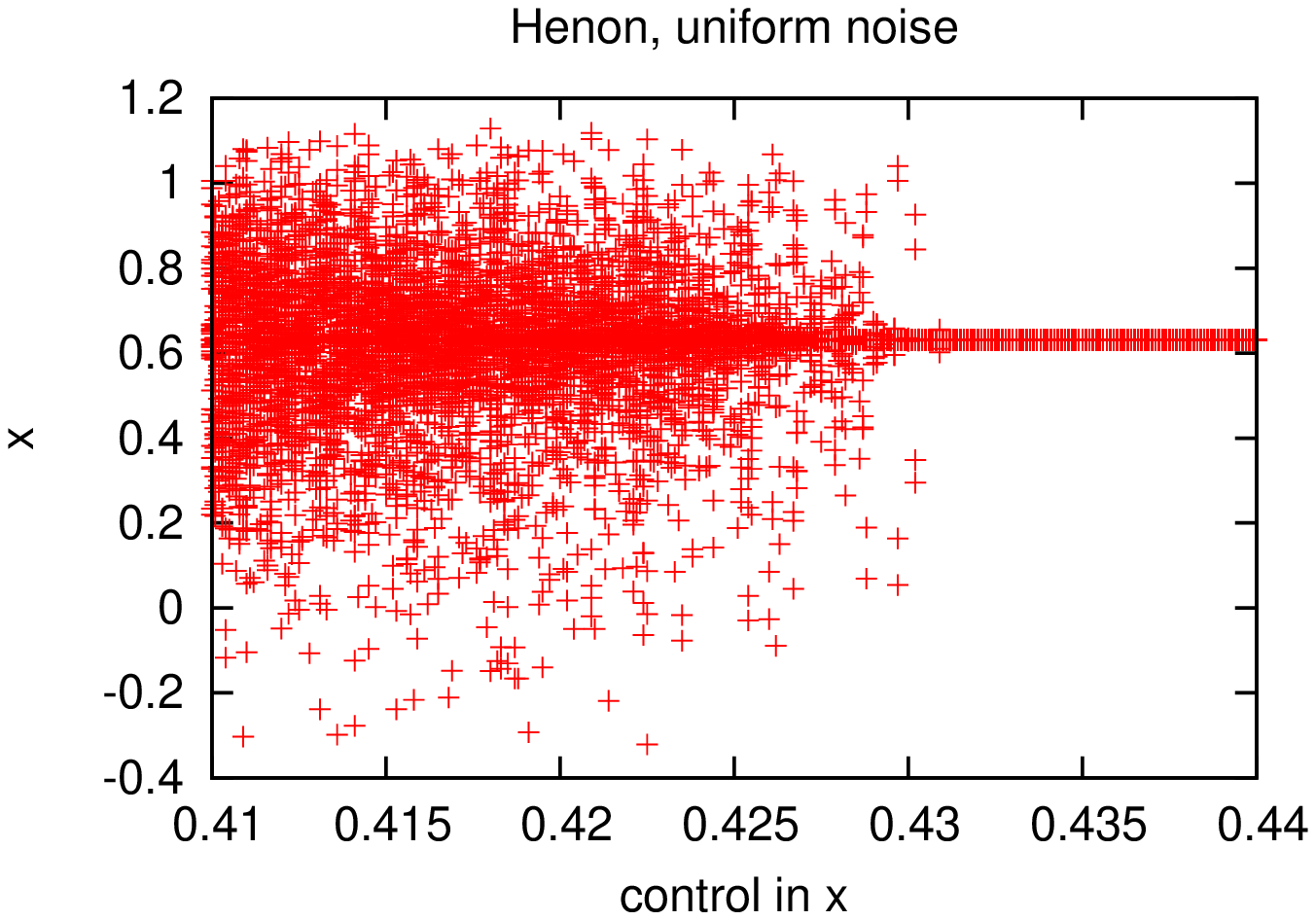}}}
\caption{Bifurcation diagram of the H\'{e}non map for $x$ with  no control in $y$ and (a) 
no noise $\ell_1=0$, (b) $\ell_1=0.2861$, Bernoulli distribution, (c)
$\ell_1=0.2861$, uniform on $[-1,1]$ distribution, $\alpha_2=0.8$. The range of initial values is $x_0 \in [0.1,0.8]$ and $y_0 \in [0.1,0.2]$.
\bigskip
}
\label{figure_ex_2_fig_1}
\end{figure}

The fact that convergence is observed for lower control values than theoretically predicted, is also illustrated in Fig.~\ref{figure_ex_2_fig_2} with limit two-dimensional sets for the H\'{e}non map, with $\ell_1=0.2861$, and the control in $y$ being $\alpha_2=0.9$,
starting with $\alpha_1=0.3$ and the Bernoulli noise (Fig.~\ref{figure_ex_2_fig_2} (a)), then for $\alpha_1=0.3$ and the uniform noise (Fig.~\ref{figure_ex_2_fig_2} (b)), next for $\alpha_1=0.36$ and the Bernoulli noise (Fig.~\ref{figure_ex_2_fig_2} (c))
and, finally, for $\alpha_1=0.36$ and the uniform noise (Fig.~\ref{figure_ex_2_fig_2} (d)).

%in-text figure
\begin{figure}
%\centering
\subfloat[]{%
\resizebox*{3cm}{!}{\includegraphics{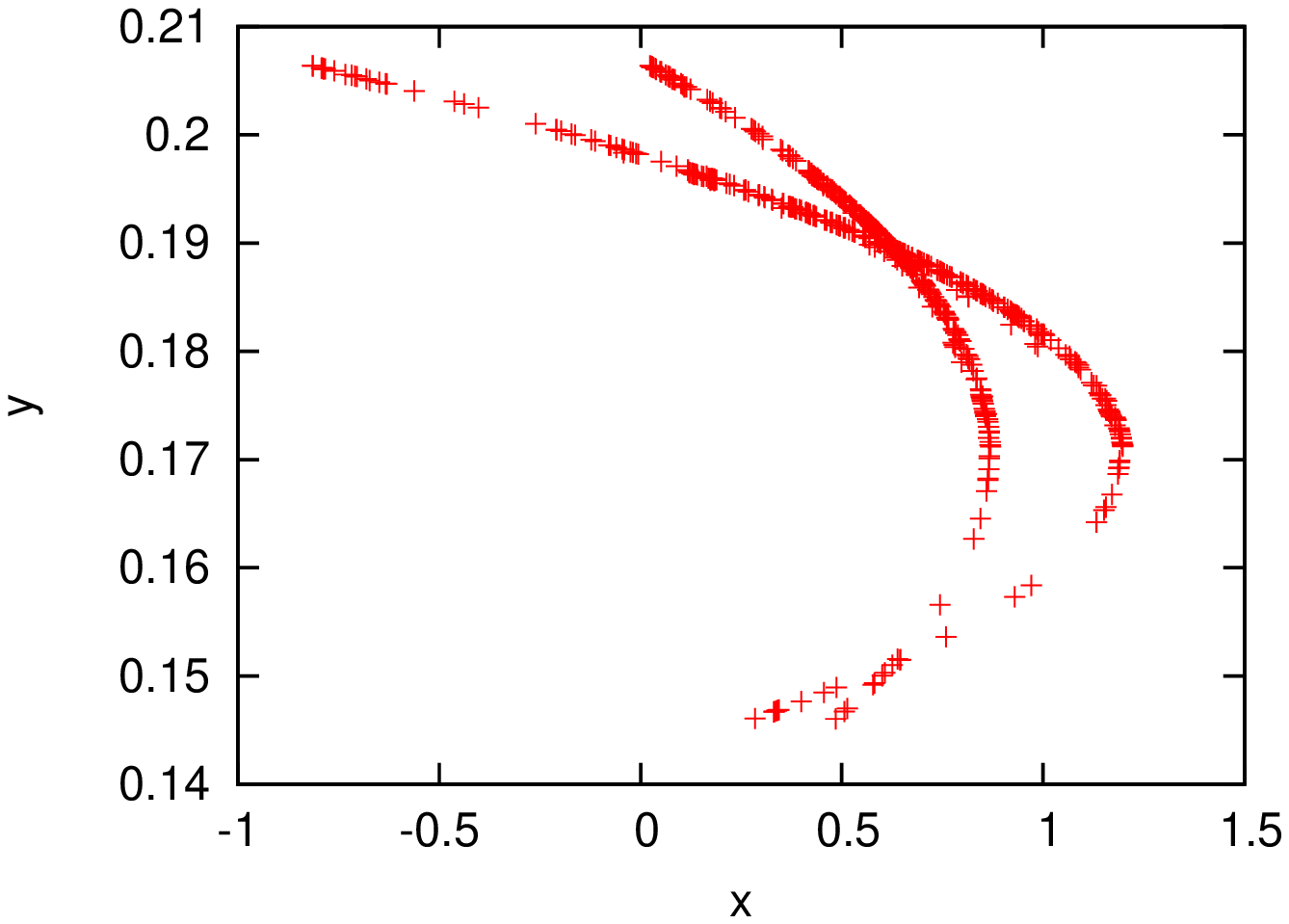}}}
\hspace{5pt}
~~~
\subfloat[]{%
\resizebox*{3cm}{!}{\includegraphics{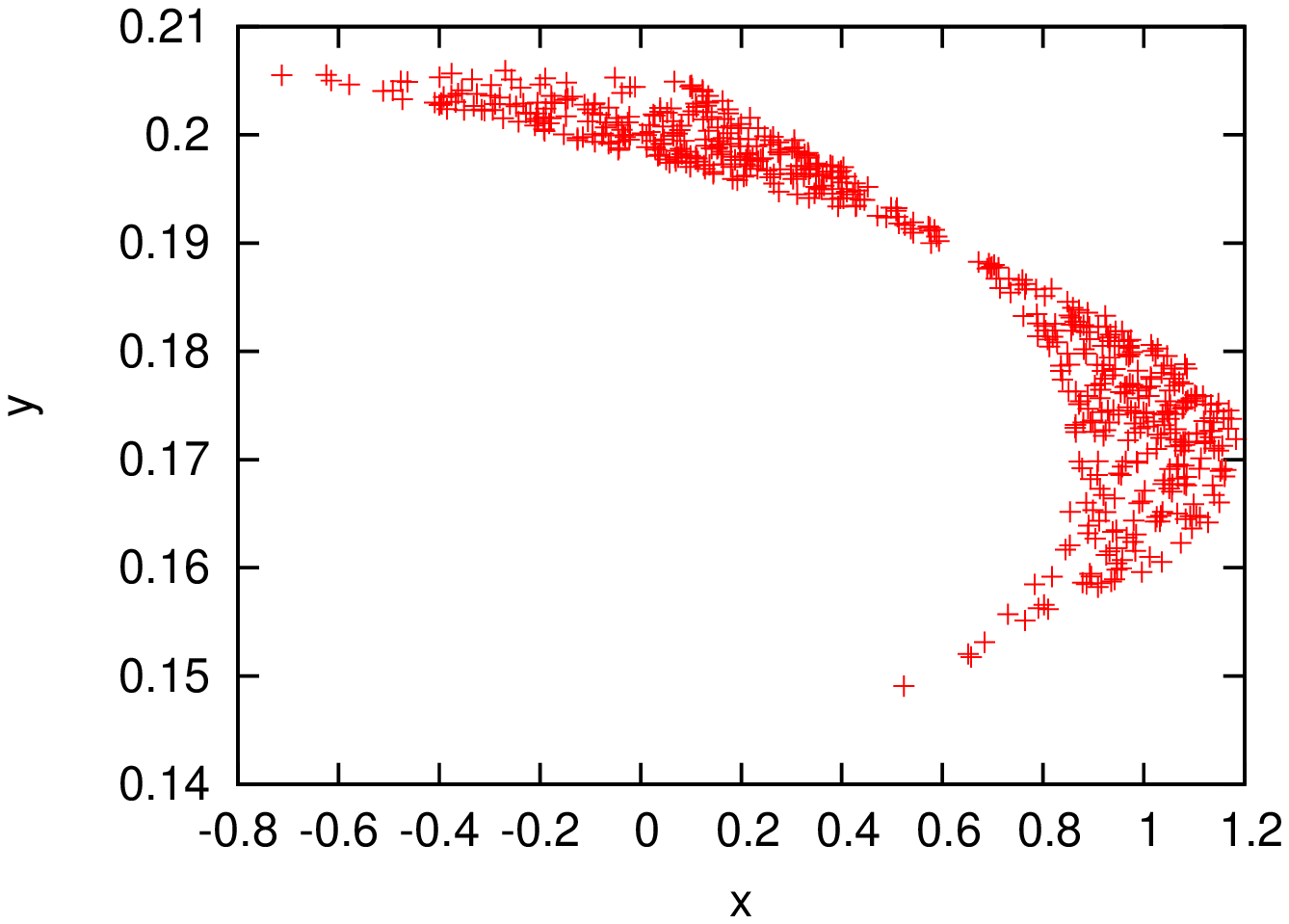}}}
~~~
\subfloat[]{
\resizebox*{3cm}{!}{\includegraphics{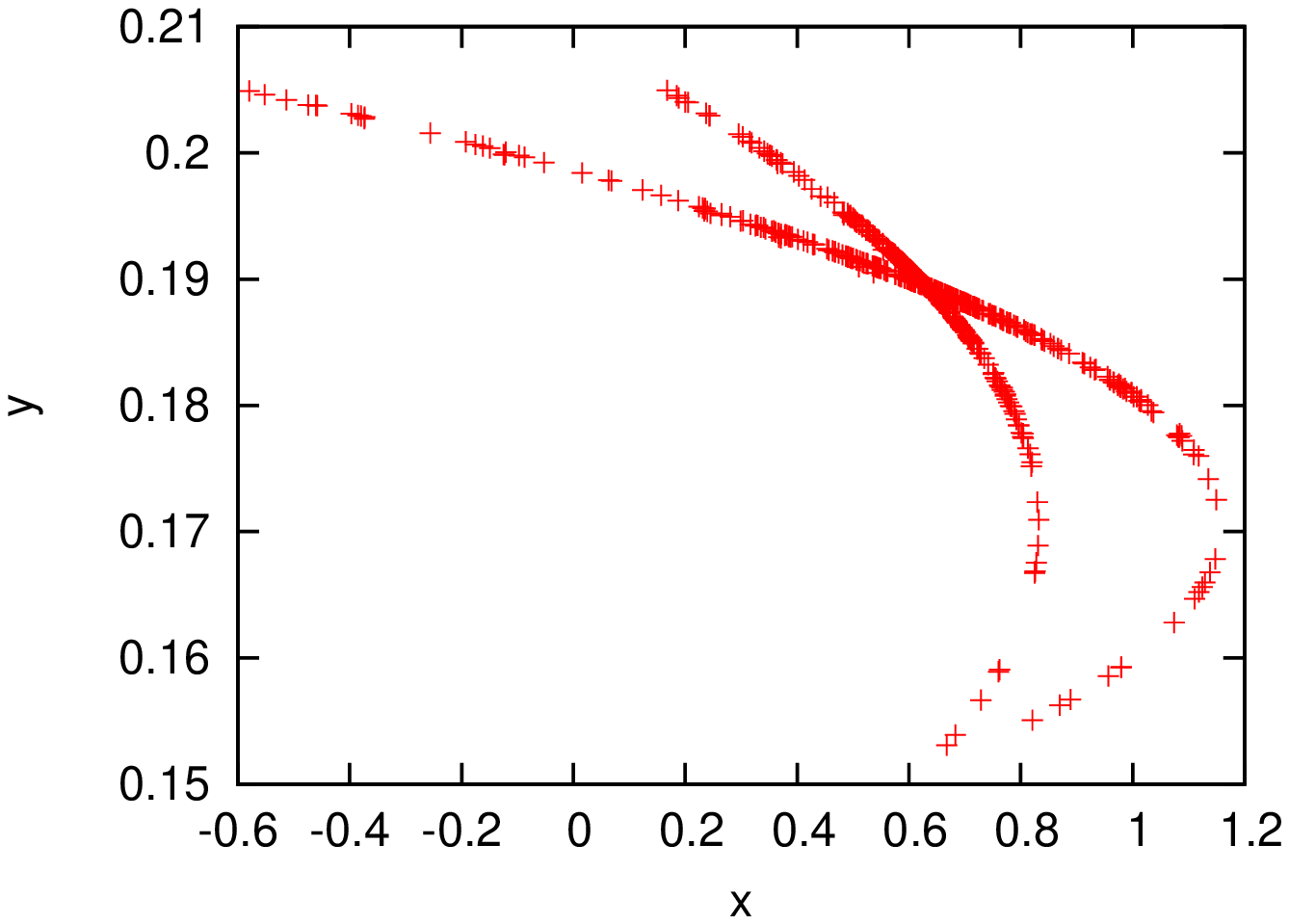}}}
\hspace{5pt}
~~~
\subfloat[]{%
\resizebox*{3cm}{!}{\includegraphics{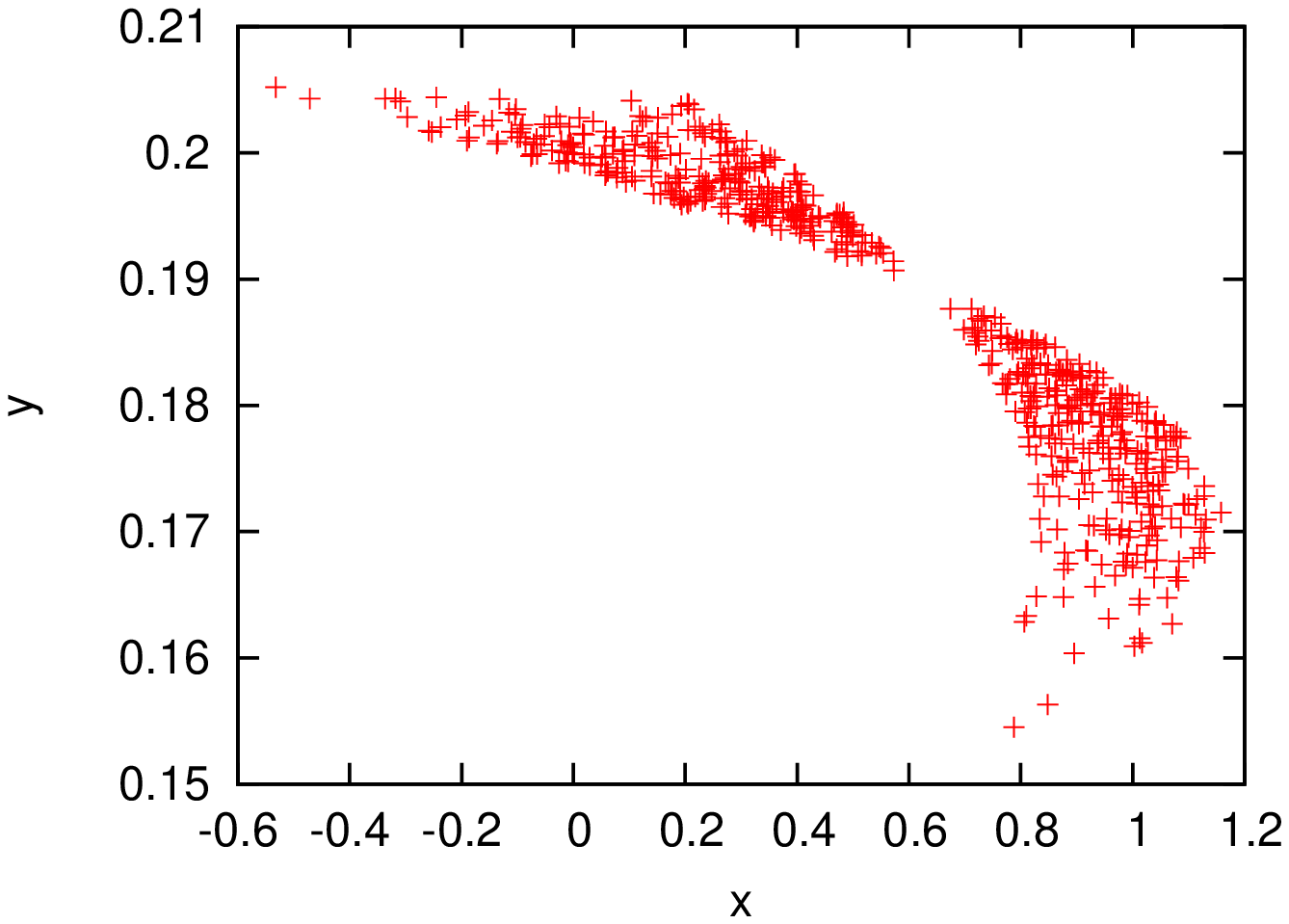}}}
\caption{The limit sets for the H\'{e}non map for  $\ell_1 = 0.2861$ and  $\alpha = 0.3$, (a) for 
Bernoulli, (b) for the uniform on $[-1,1]$ noise, then $\alpha = 0.36$, (c) for 
Bernoulli, (d) for the uniform on $[-1,1]$ noise. The control in
$y$ is 0.9, $x_0 = 0.3$, $y_0=0.1$ in all four simulations.
\bigskip
}
\label{figure_ex_2_fig_2}
\end{figure}

\item  
{\it Norm $\| \cdot\|_1$, $\chi_1$ is Bernoulli distributed, $\ell_2=0$, $a=2$ and $b=0.5$. }  In this part we take  different parameters $a$ and $b$,  which represent a chaotic case. In this case $x^* \approx 0.593$, $2ax^* \approx 2.372$ and 
\begin{equation*}
\begin{split}
\| \mathcal C_{n}(x)\|_\infty \le & \max\left\{\left(1- \alpha_1 - \ell_1 \chi_{1, n} \right)2ax^*+\left( 1-\alpha_2 - \ell_2 \chi_{2, n} \right)  b, \right. \\ 
& \left.  \left(1- \alpha_1 - \ell_1 \chi_{1, n} \right) \right\} \\
& \approx 2.372\left( 1-\alpha_1 - \ell_1 \chi_{1, n} \right)+0.5\left( 1-\alpha_2 - \ell_2 \chi_{2, n} \right)=:\nu(n).
\end{split}
\end{equation*}
This leads to the estimate $\ell^2_1>\left(( 1-\alpha_1)+0.2107( 1-\alpha_2)\right)^2-0.1777$, which, along with  \eqref{def:d}, holds
for $\alpha_1=0.45$, $\alpha_2=0.8$ and 
$\ell_1=0.416$.

Note that in the deterministic case, when $\ell_1=\ell_2=0$,  we get the best value (using the estimation of eigenvalues) as $\alpha_1>0.6518$ for $\alpha_2=0$ and $\alpha_1>0.5801$ for $\alpha_2=0.8$.
\end{enumerate}
%
%%%%%%%%%%%%%%%%%%%%%%%%%%%%%

\subsubsection{\bf Lozi}
\label{subsec:stochLozi}

\begin{enumerate}
\item{\it Norm $\|\cdot\|_\infty$, $\chi_1$ is Bernoulli distributed.} Reasoning as for the H\'{e}non map we get
\begin{equation*}
\begin{split}
\| \mathcal C_{n}(x)\|_\infty \le & \max\{2.4\left( 1-\alpha_1 -\ell_1 \chi_{1, n} \right),  0.3\left( 1-\alpha_2 - \ell_2 \chi_{2, n} \right) \} \\ = & 2.4\left( 1-\alpha_1 -\ell_1 \chi_{1, n} \right)=:\nu(n)
\end{split}
\end{equation*}
if
\begin{equation}
\label{ineq:a122}
  1-\alpha_2 + \ell_2\le 8\left(1- \alpha_1 - \ell_1 \right).
 \end{equation}
Then $\mathbb E \ln \nu(n)=\ln  2.4+0.5\ln \left( (1-\alpha_1)^2- \ell_1^2\right)< 0$ leads to the estimate $\ell_1^2>(1-\alpha_1)^2-0.1736$. 
The values $\alpha_1=0.414$, $ \ell_1=0.17$ satisfy \eqref{def:d}  as well as  \eqref{ineq:a122} (for arbitrary $\alpha_2$, $\ell_2$).

Simulations illustrate convergence for smaller values of  $\alpha_1$ and $\ell_1$ for 
$\alpha_2=\ell_2=0$.
In Fig.~\ref{figure_ex_1_fig_5} (d), for $\alpha_1 = 0.4<0.413$ and $\ell_1=0.15 < 0.17$, there is convergence to the equilibrium; however, such convergence is achieved for any $\ell_1 \geq 0.15$. According to Fig.~\ref{figure_ex_1_fig_5} (a), for $\ell_1=0$ there is a stable two-cycle, which becomes blurred and approaches the equilibrium for $\ell_1=0.05$ (Fig.~\ref{figure_ex_1_fig_5} (b)), leading to a blurred equilibrium at $\ell_1=0.1$ (Fig.~\ref{figure_ex_1_fig_5} (c)) and a stable equilibrium for $\ell_1=0.15$ (Fig.~\ref{figure_ex_1_fig_5} (d)). 
In Fig.~\ref{figure_ex_1_fig_5}, we took the initial point $(-10,-15)$ quite far from the equilibrium;  the same stabilization pattern is observed for any initial point. 
The only difference is that for very large by absolute value $x_0 \approx -100$, 
up to 70 first iterations can be outside of the interval $[-1;1.5]$ chosen to illustrate convergence.

%in-text figure
\begin{figure}
%\centering
\subfloat[]{%
\resizebox*{3cm}{!}{\includegraphics{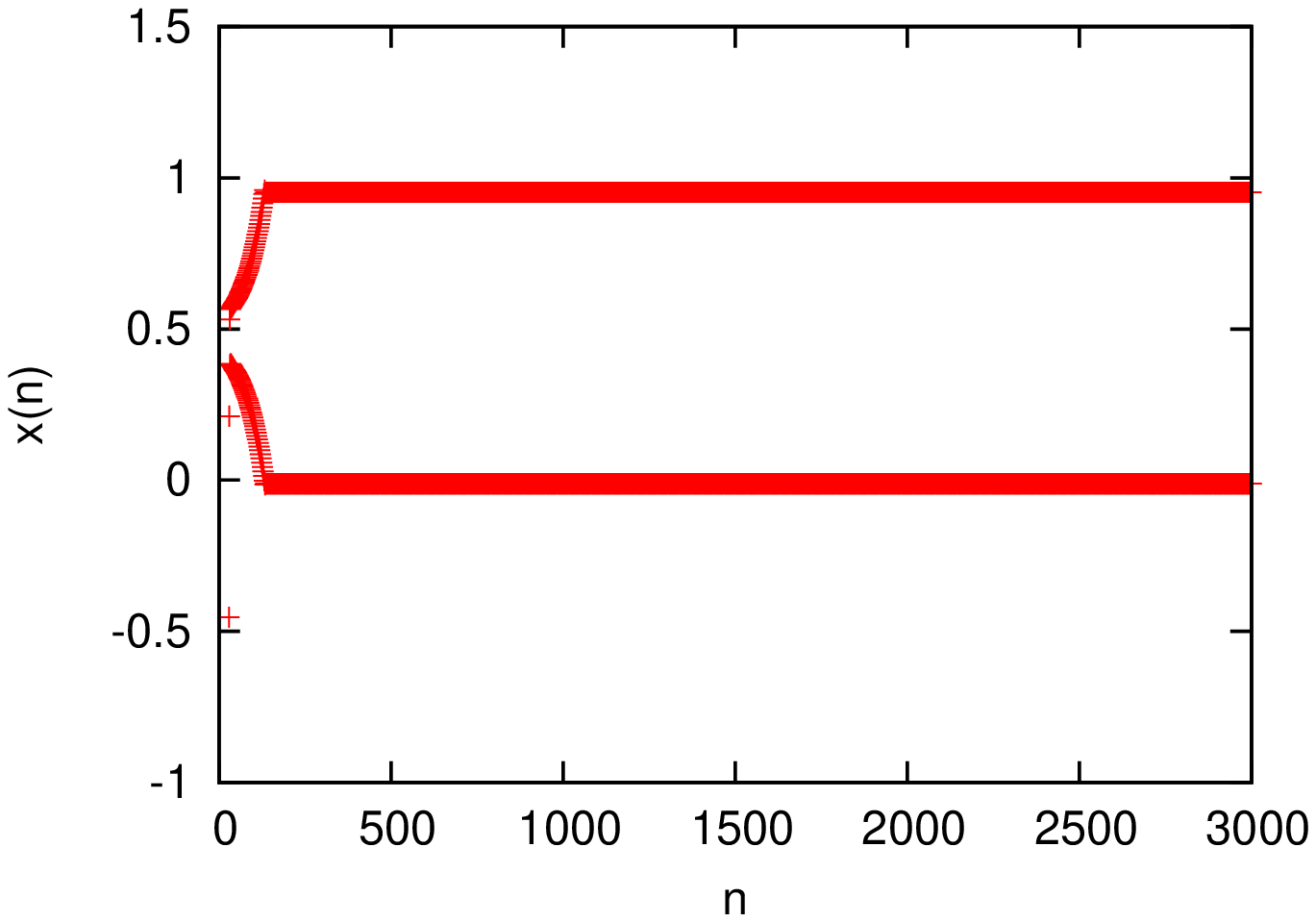}}}
\hspace{5pt}
~~~
\subfloat[]{%
\resizebox*{3cm}{!}{\includegraphics{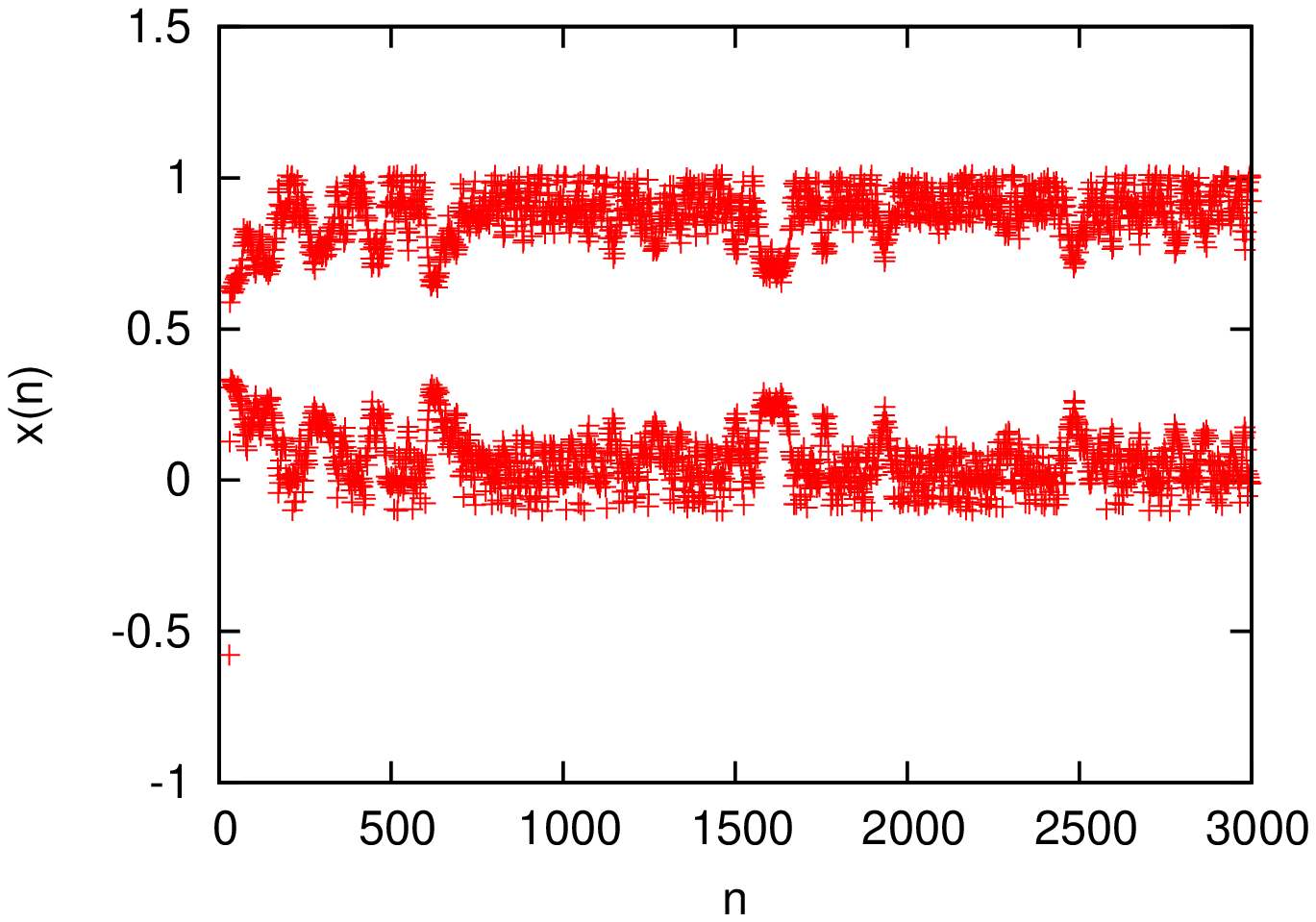}}}
~~~
\subfloat[]{
\resizebox*{3cm}{!}{\includegraphics{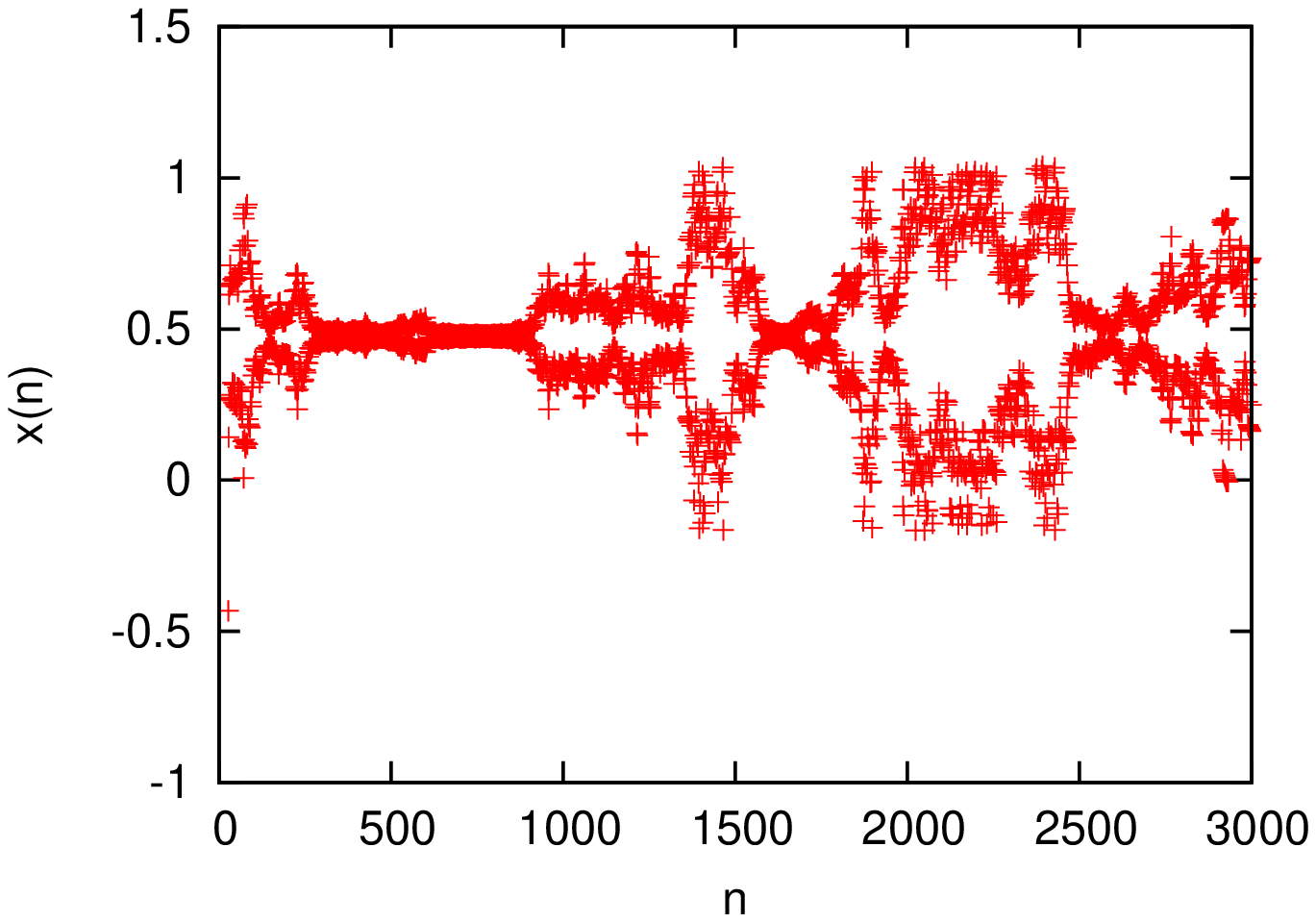}}}
\hspace{5pt}
~~~
\subfloat[]{%
\resizebox*{3cm}{!}{\includegraphics{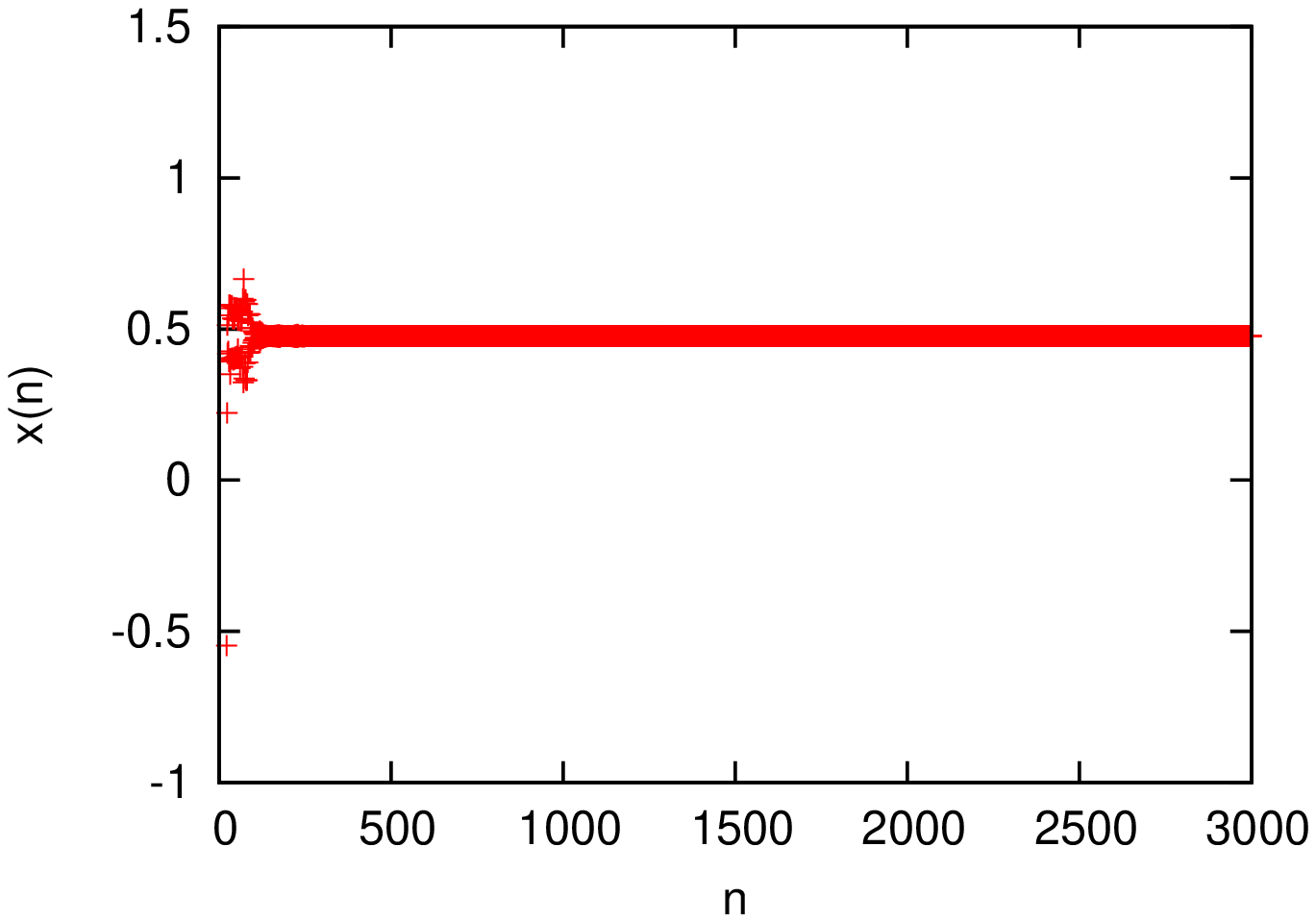}}}
\caption{Runs of the Lozi map for $x$-coordinate only, $\alpha = 0.4$, no control in
$y$, $x_0=-10$, $y_0=-15$ and (a) $\ell = 0$, (b) $\ell = 0.05$, 
(c) $\ell = 0.1$, (d) $\ell = 0.15$ for the Bernoulli noise.
\bigskip
}
\label{figure_ex_1_fig_5}
\end{figure}

In addition to separate runs in $x$, asymptotic behavior is illustrated by two-dimensional diagrams of the limit sets in Figure~\ref{figure_ex_1_fig_6}.
This set forms a blurred two-cycle for  $\ell_1=0.03$ (Fig.~\ref{figure_ex_1_fig_6} (a)) which expands and approached the equilibrium point for $\ell_1=0.08$ (Fig.~\ref{figure_ex_1_fig_6} (b)), becomes a blurred equilibrium for $\ell_1=0.135$ (Fig.~\ref{figure_ex_1_fig_6} (c)) and a stable equilibrium for $\ell_1=0.145$ (Fig.~\ref{figure_ex_1_fig_6} (d)) or larger control values.

%in-text figure
\begin{figure}
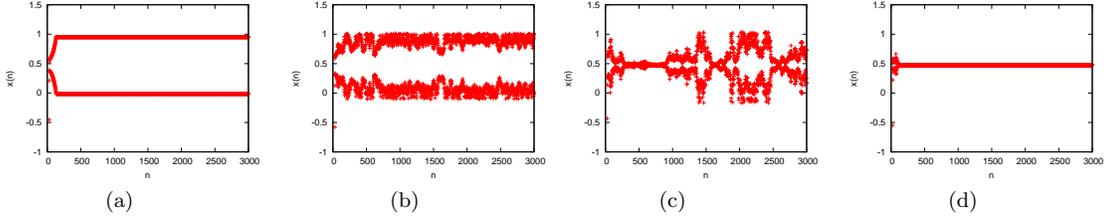

%\centering
\subfloat[]{%
\resizebox*{3cm}{!}{\includegraphics{lo_x_c_0_4_no_y_c_no_l_x0_m10_y_0_m15.eps}}}
\hspace{5pt}
~~~
\subfloat[]{%
\resizebox*{3cm}{!}{\includegraphics{lo_x_c_0_4_no_y_c_l_0_05_x0_m10_y_0_m15.eps}}}
~~~
\subfloat[]{
\resizebox*{3cm}{!}{\includegraphics{lo_x_c_0_4_no_y_c_l_0_1_x0_m10_y_0_m15.eps}}}
\hspace{5pt}
~~~
\subfloat[]{%
\resizebox*{3cm}{!}{\includegraphics{lo_x_c_0_4_no_y_c_l_0_15_x0_m10_y_0_m15.eps}}}
\caption{The limit set for the Lozi map for  $\alpha = 0.4$, no control in
$y$, $x_0 = -10$, $y_0=-15$ and (a) $\ell = 0.03$, (b) $\ell = 0.08$,
(c) $\ell = 0.135$, (d) $\ell = 0.145$ for the Bernoulli 
noise.
\bigskip
}
\label{figure_ex_1_fig_6}
\end{figure}

\item{\it Norm $\| \cdot\|_1$, $\chi_1$ is Bernoulli distributed, $\ell_2=0$. }
Now, 
\begin{equation*}
\begin{split}
\| \mathcal C_{n}(x)\|_1&\le 1.4\left( 1-\alpha_1 - \ell_1 \chi_{1, n} \right)+0.3\left( 1-\alpha_2 - \ell_2 \chi_{2, n} \right)=:\nu(n),\\
\mathbb E \ln \nu(n)&=0.5\ln \left[ \left(1.4( 1-\alpha_1)+0.3( 1-\alpha_2)\right)^2-\left(1.4\ell_1\right)^2\right].
\end{split}
\end{equation*}
We can show that  each set of parameters $\alpha_1=0.3$, $\alpha_2=0.8$, $\ell_1=0.2039$ and $\alpha_1=0.27$,  $\alpha_2=0.9$,  $\ell_1=0.2332$, satisfies all necessary conditions and leads to $\mathbb E \ln \nu(n)<0$.

\item{\it Norm $\| \cdot \|_1$, $\chi_1$ and $\chi_2$ are  Bernoulli distributed.}
We have 
\begin{equation*}
\begin{split}
\mathbb E \ln \nu(n)&=0.25\biggl(\ln \left[ \left(1.4( 1-\alpha_1)+0.3( 1-\alpha_2)\right)^2-\left(1.4\ell_1+0.3\ell_2\right)^2\right]\\&+\ln \left[ \left(1.4( 1-\alpha_1)+0.3( 1-\alpha_2)\right)^2-\left(1.4\ell_1-0.3\ell_2\right)^2\right]\biggr).
\end{split}
\end{equation*}
By introduction of the second nonzero noise, we can decrease a little bit the intensity $\ell_1$ of the first noise for  $\alpha_1=0.27$, $\alpha_2=0.9$. In particular, we can take $\ell_1=0.2$, $\ell_1=0.55$ and get $\mathbb E \ln \nu(n)\approx 0.25\ln 0.9936 \approx -0.25 \cdot 0.0064<0$.

Fig.~\ref{figure_ex_3_fig_1} compares the influence of Bernoulli and uniformly distributed on $[-1,1]$ noises
for the Lozi map, illustrating a part of the bifurcation diagram for $x$ in the range of parameters of $\alpha$ close to the last bifurcation leading to stabilization of the equilibrium, with the control $\alpha_2=0.9$ in $y$. In Fig.~\ref{figure_ex_3_fig_1} (a), for $\ell_1=0$, we see a period-halving bifurcation; in (b) and (c) (Bernoulli and uniform noise, respectively, both with the amplitude $\ell_1=0.2$), the value of $\alpha$ for this bifurcation is smaller, with some stabilization advantage of the Bernoulli distribution. Fig.~\ref{figure_ex_3_fig_1} (c) also shows that the image for the uniformly distributed perturbation looks `noisier' (compared to (b)).

%in-text figure
\begin{figure}
%\centering
\subfloat[]{%
\resizebox*{3.5cm}{!}{\includegraphics{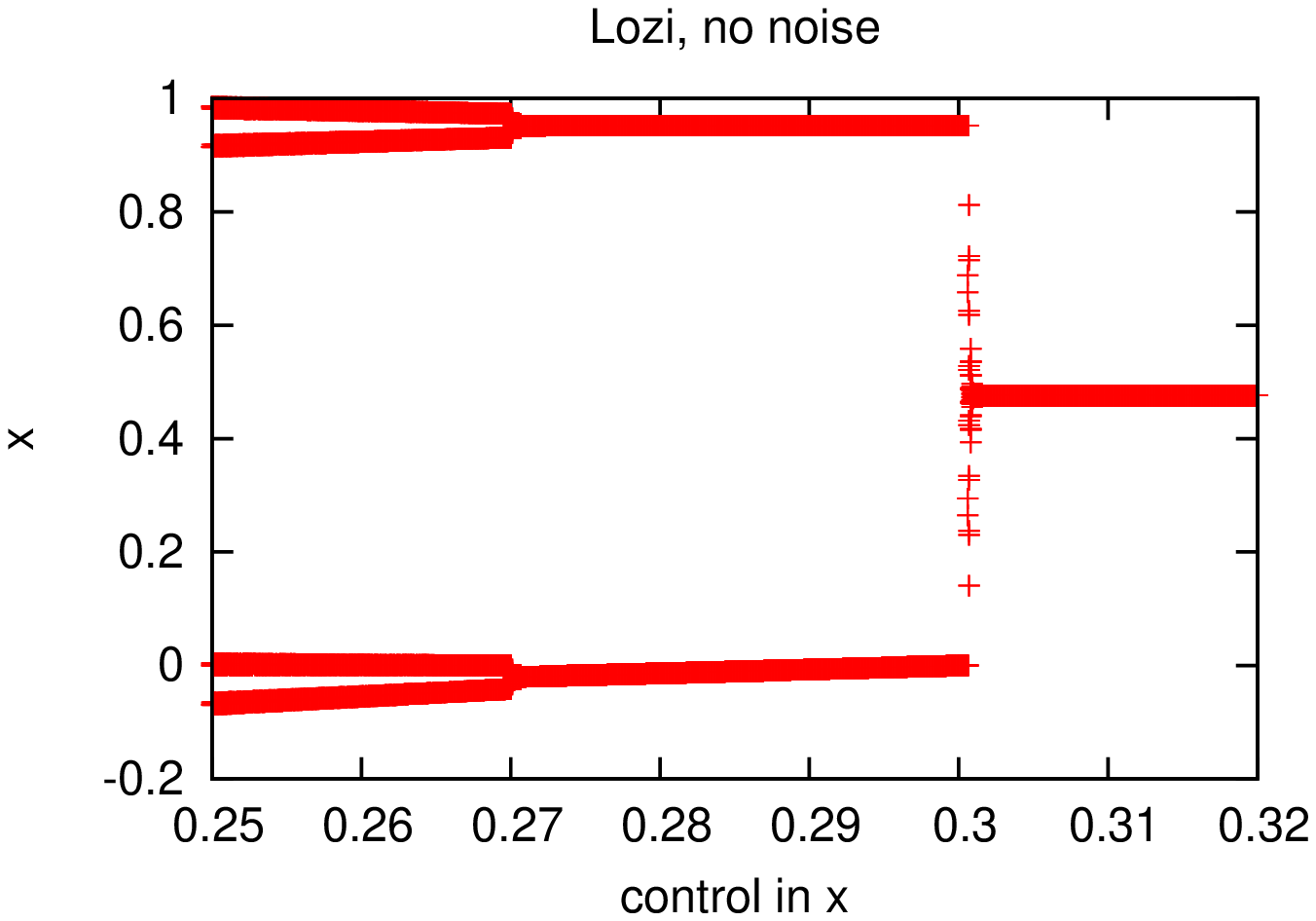}}}
\hspace{12pt}
~~~
\subfloat[]{%
\resizebox*{3.5cm}{!}{\includegraphics{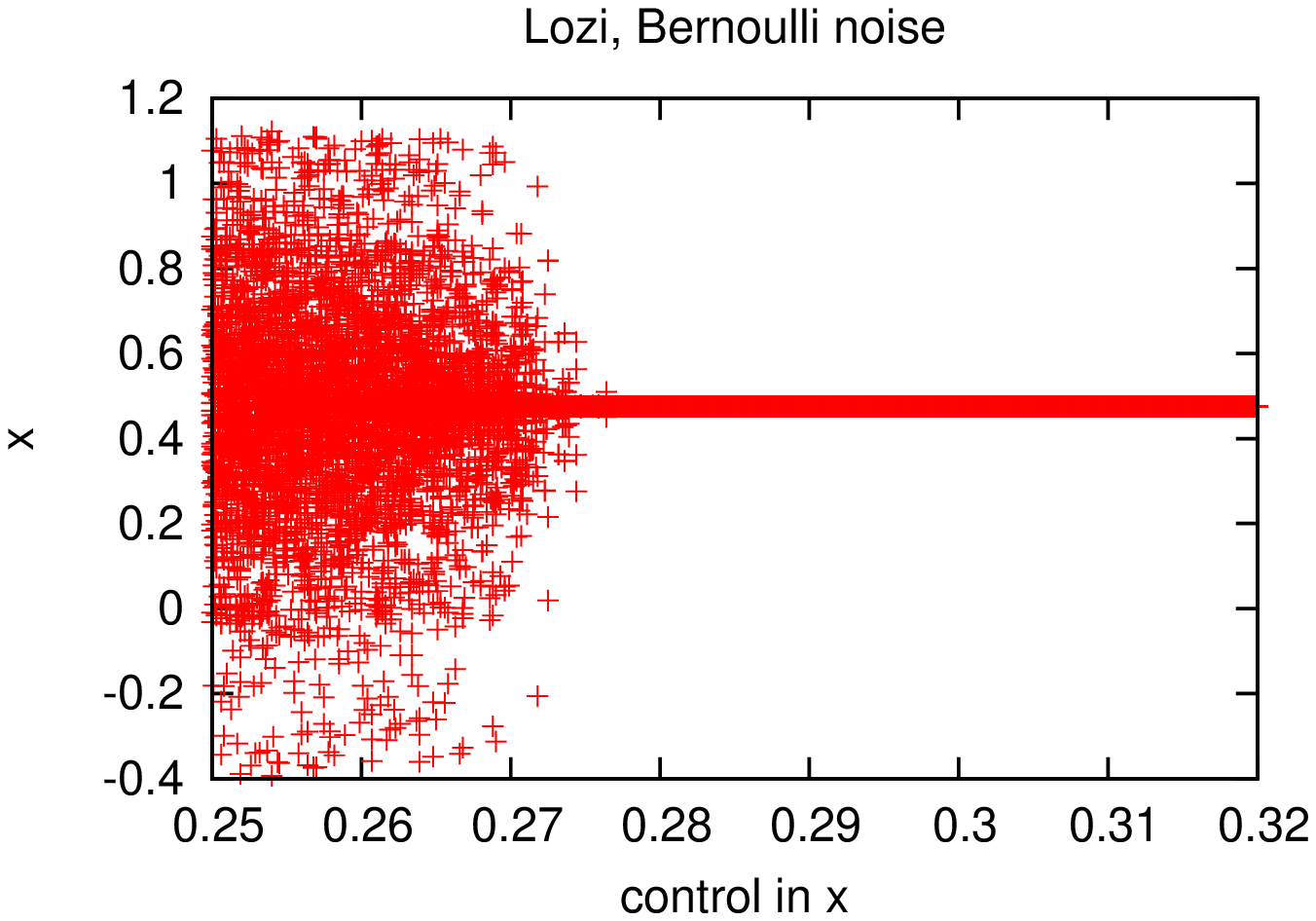}}}
\hspace{12pt}
~~~
\subfloat[]{
\resizebox*{3.5cm}{!}{\includegraphics{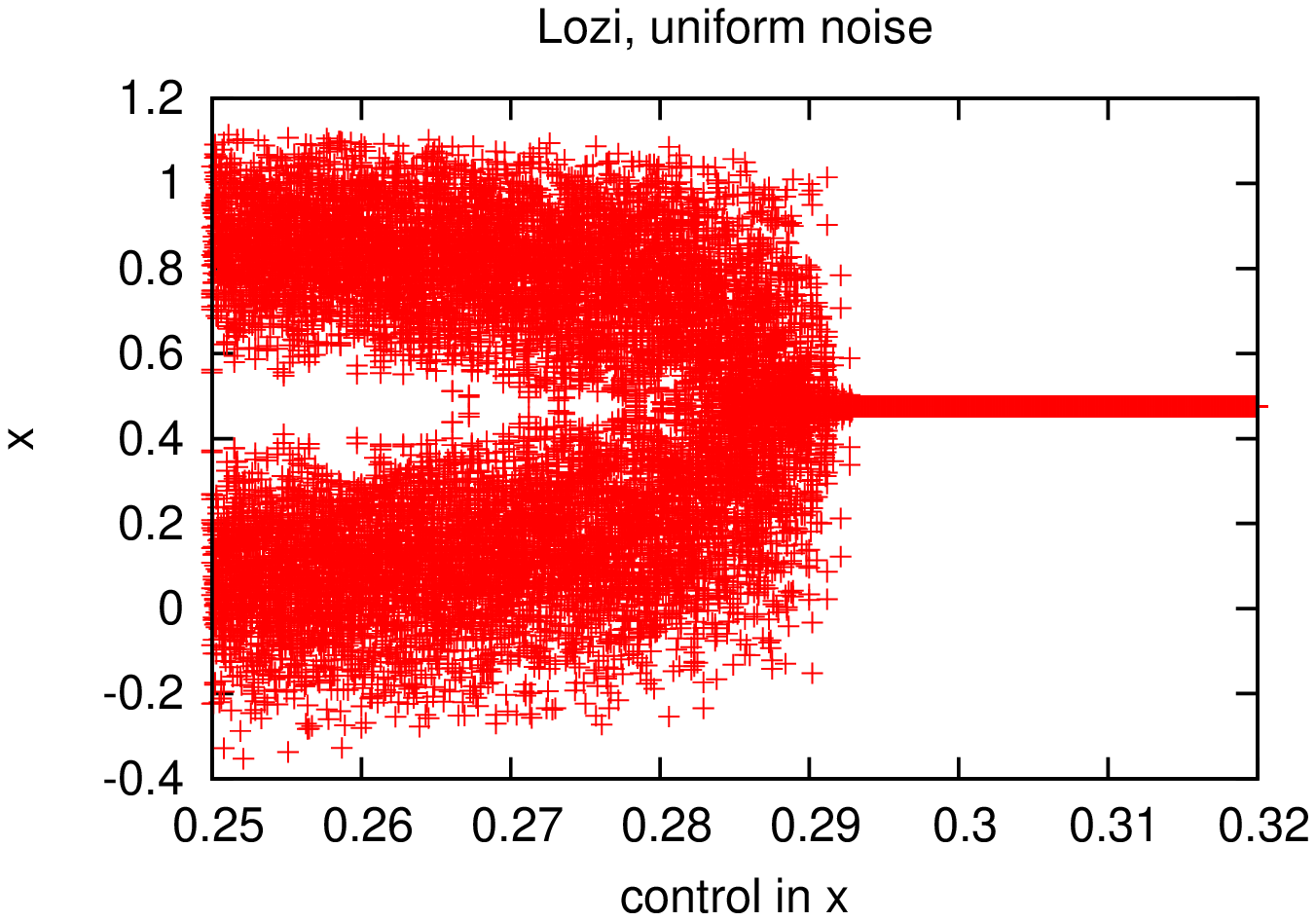}}}
\caption{Bifurcation diagram of the Lozi map for $x$ with  0.9-control in $y$ and (a)
no noise $\ell_1=0$, (b) $\ell_1=0.2$, Bernoulli distribution, (c)
$\ell_1=0.2$, uniform on $[-1,1]$ noise distribution. 
The range of initial values is $x_0 \in [0.1,0.8]$ and $y_0 \in [0.1,0.2]$.
\bigskip
}
\label{figure_ex_3_fig_1}
\end{figure}

Finally, we illustrate the role of noise $\ell_2$ in stabilization. 
If $\alpha_1 = 0.27$, $\alpha_2 = 0.9$, without noise, stabilization is not observed (Fig.~\ref{figure_ex_1_fig_7} (a)). 
Even with $\ell_1=0.4$ and $\ell_1=0.5$, there is no stabilization, see Fig.~\ref{figure_ex_1_fig_7} (b) and (c). 
However, if $\ell_2$ increases to 0.55, stabilization is observed, 
as theoretically predicted, see Fig.~\ref{figure_ex_1_fig_7} (d).

\begin{figure}
%\centering
\subfloat[]{%
\resizebox*{3cm}{!}{\includegraphics{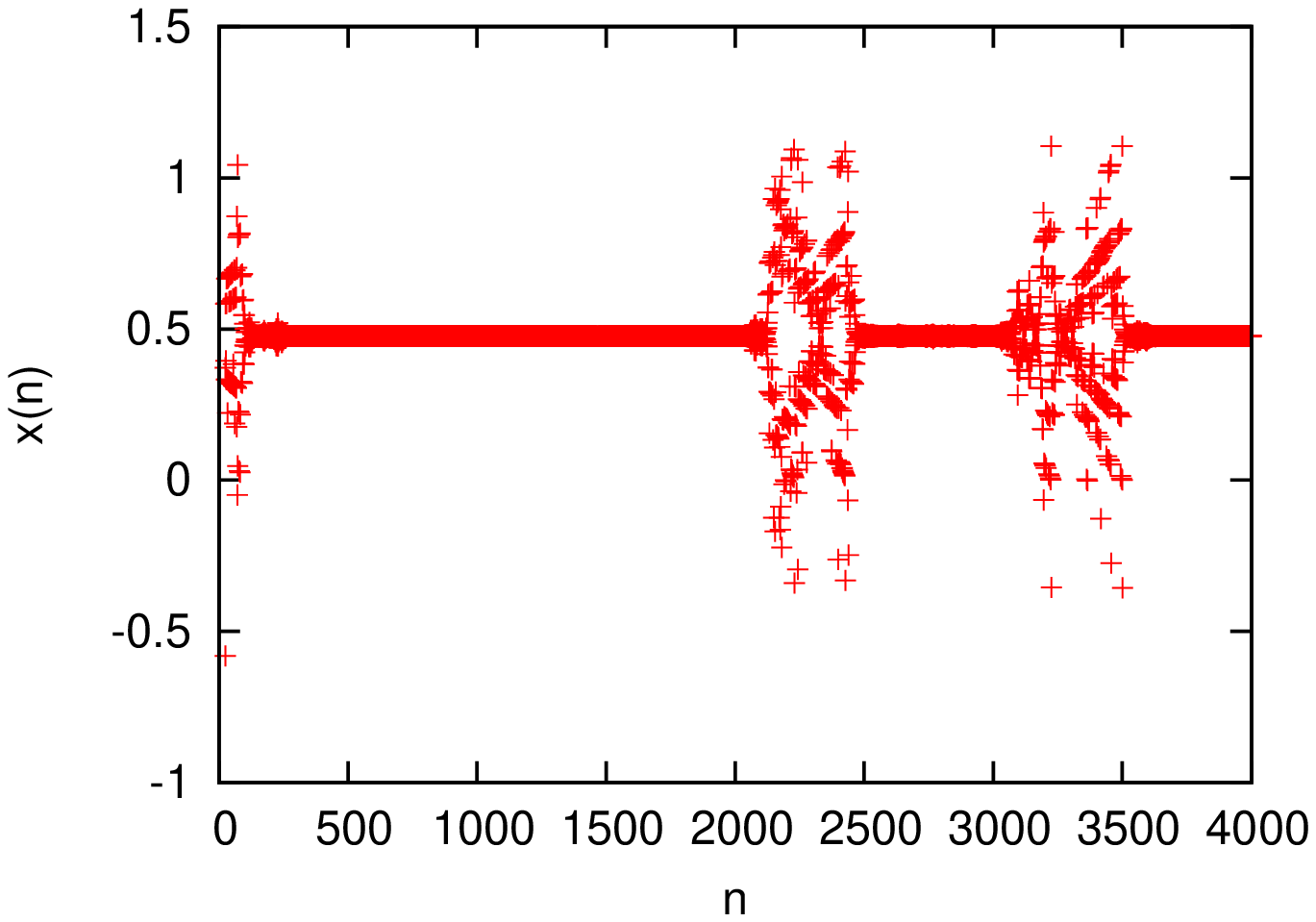}}}
\hspace{5pt}
~~~
\subfloat[]{%
\resizebox*{3cm}{!}{\includegraphics{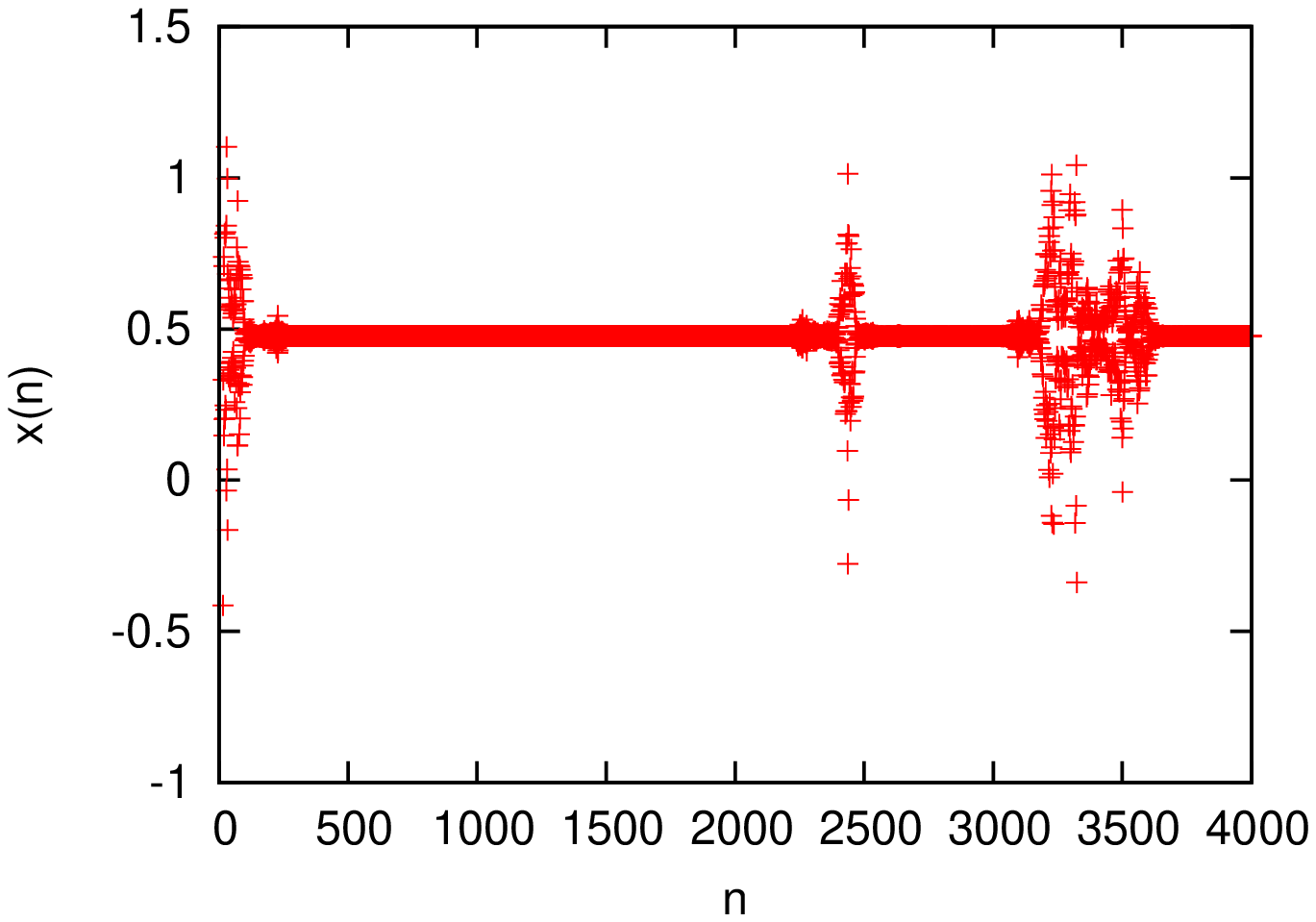}}}
~~~
\subfloat[]{
\resizebox*{3cm}{!}{\includegraphics{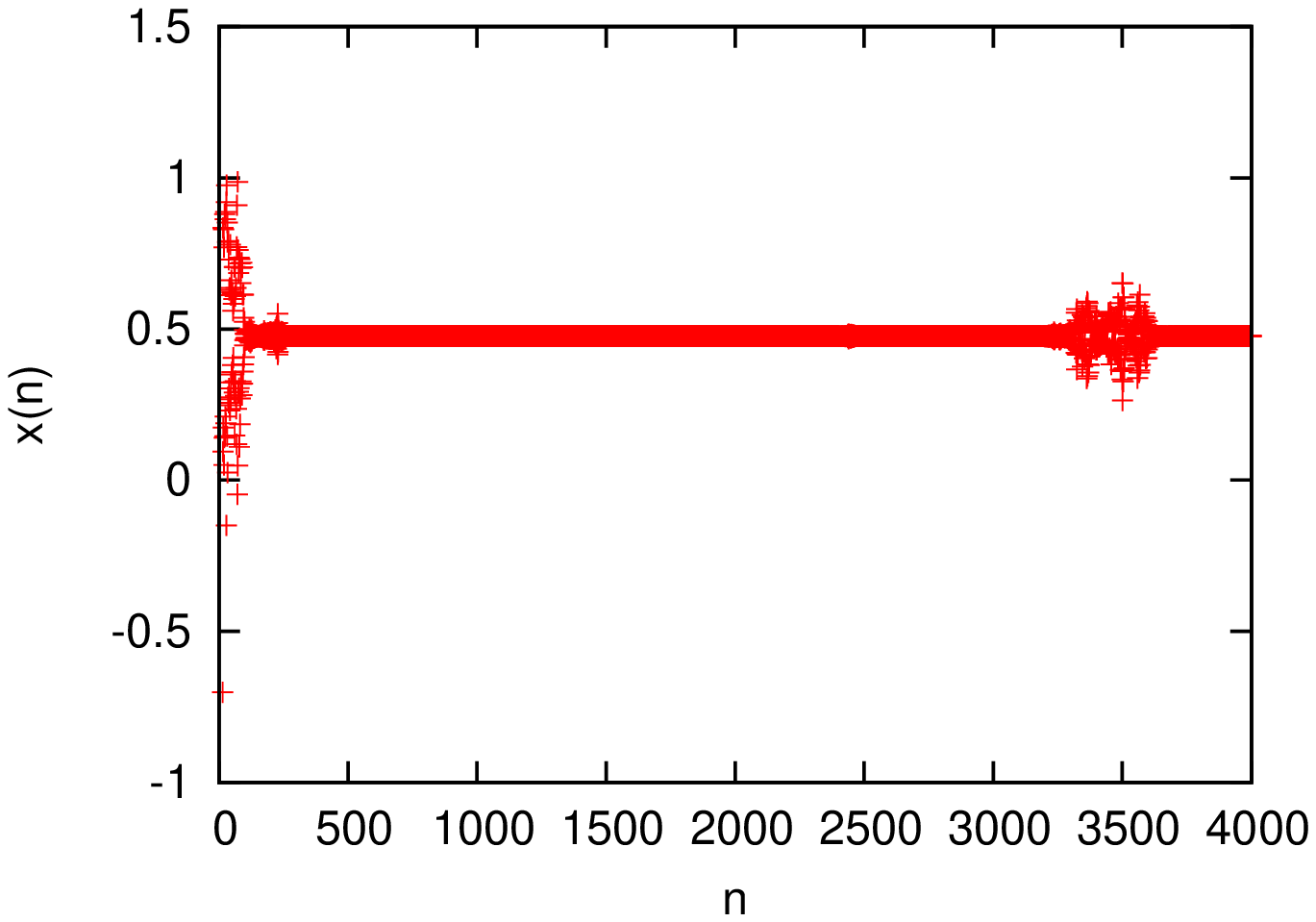}}}
\hspace{5pt}
~~~
\subfloat[]{%
\resizebox*{3cm}{!}{\includegraphics{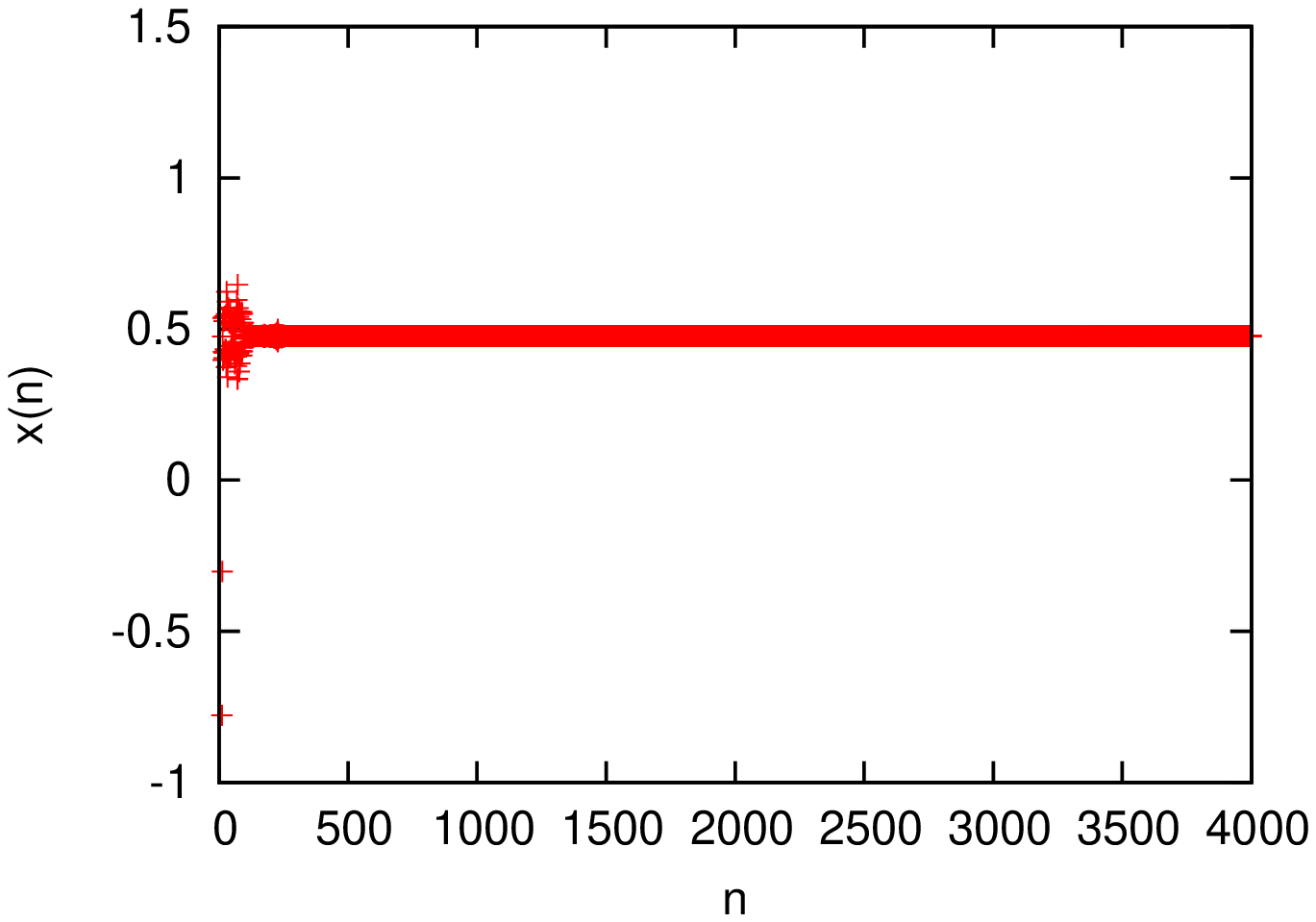}}}
\caption{Runs of the Lozi map for $x$-coordinate only, $\alpha_1 = 0.27$, $\ell_1=0.2$, the control in
$y$ is $\alpha_2 = 0.9$, $x_0=-10$, $y_0=-15$ and %(from left to right) 
(a) $\ell_2 = 0$, (b) $\ell_2 = 0.4$, (c) $\ell_2 = 0.5$, (d) $\ell_2 = 0.55$ for the Bernoulli
noise for both $\ell_i$.
\bigskip
}
\label{figure_ex_1_fig_7}
\end{figure}

\end{enumerate}

%%%%%%%%%%%%%%%%%%%%%
\section{Conclusions and Discussion}
\label{sec:conclusions}

In the present paper, we explored both deterministic and stochastic stabilization 
of H\'{e}non and Lozi maps with TOC, where an equilibrium is chosen as a target.
\begin{enumerate}
\item
We proved that TOC allows to stabilize a chosen equilibrium in a ball centered at it, once controls are close enough to one. For an autonomous model,
simulations illustrate, however, that local stabilization parameters work in a reasonably large neighborhood. This brings us to the problem whether there is an equivalence of local and global stability observed in certain one-dimensional cases
of TOC \cite{TPC}. However, for both H\'{e}non and Lozi maps it is still an open question, even if we reduce the domain to the same half-plane where the equilibrium lies, and consider autonomous cases only.
\item
Investigation of non-autonomous models makes it possible to add a noise in certain bounds, still keeping stabilization. Moreover, we demonstrate theoretically and with simulations that stochastic stabilization may be achieved when there is, e.g. a stable cycle rather than a stable equilibrium for the average controls values.
\end{enumerate}

Let us compare our results to other types of control. 

Prediction-Based Control (PBC) keeps all the original equilibrium points.
The general form of PBC applied to a scalar map $x_{n+1}=f(x_n)$ is \cite{uy99}
$$x_{n+1}=f(x_n)-\alpha  \left( f^k(x_n)-x_n \right), \quad x_0>0, \quad n\in {\mathbb N}_0,
$$
which can be treated as a discrete analogue of delayed control for continuous models \cite{Pyragas1992}.
Here the control $\alpha \in (0,1)$, while $f^k$ is the $k$th iteration of $f$. The application and the analysis become simpler for $k=1$
\cite{FL2010}
\begin{equation*}
%\label{eq:intr2}
x_{n+1}=f(x_n)- \alpha (f(x_n)-x_n)= (1-\alpha)f(x_n)+ \alpha x_n, \quad x_0>0, \quad n\in {\mathbb N}_0.
\end{equation*}
However, as mentioned in \cite{LP2014}, immediate generalization of this scheme to the vector case
$$
X_{n+1} = F(X_n) - \alpha(F(X_n)-X_n)
$$
does not stabilize the H\'{e}non map for any $\alpha \in (0,1)$. A modification of PBC proposed in \cite{Polyak} was applied to get local stabilization of both equilibrium points for the H\'{e}non map in \cite{LP2014}.

Compared to these methods, the advantage of VMTOC is robustness allowing stabilization in a prescribed domain, not just a small neighbourhood of the equilibrium point. Introduction of several parameters allows more adaptivity, controlling, for example, one of two variables only. Also, \cite{LP2014,Polyak} considered constant control only. Introduction of step-dependent and stochastic control allowed us to evaluate positive effect of noise on stabilization. On the other hand, PBC-type design allowed to get lower control values (`weak control' in  \cite{LP2014}) for local stabilization and could be applied at selected steps only. Note that PBC in principle can provide local stabilization only, as it stabilizes all chosen equilibrium points with the same control equation. In addition, \cite{LP2014,Polyak} considered cycles stabilization, which is not in the framework of the present research.

While it is generally hard to compare continuous and discrete controls, some ideas are similar. 
For example, in \cite{Gjur1}, as in the present paper, the control is a function of some difference, which becomes small as the solution approaches the target orbit. On the other hand, in \cite{Gjur1}, to some extent similarly 
to \cite{LP2014,Polyak},
this is the difference between the current state and some other state, past in \cite{Gjur1} and future, or predicted, in 
\cite{LP2014,Polyak}. The other common feature of \cite{Gjur1} and our research is that the distributed nature of the control 
can contribute to stabilization. However, in each realization, control distribution in \cite{Gjur1} is fixed and, for a given map and prescribed initial values, the results are deterministic. In our models, each realization of the stochastic control, generally, leads to different results. This justifies the necessity to apply Kolmogorov's Law of Large Numbers to verify stabilization.

There are some other control types, which cannot be directly compared to the results of the paper as they stabilize not the equilibria of the original map but some other points.
For example, we considered TOC with an equilibrium as a target, but it could be interesting to include other possible targets, for example,
the zero target; in this case, we study Proportional Feedback control \cite{Carmona,gm,Liz2010}
\begin{equation*}
%\label{PF_eq}
x_{n+1}=f( c x_n ), \quad c\in [0,1).
\end{equation*}
%Note that in the scalar case, existence of a non-trivial equilibrium %point is not guaranteed.
%Scalar PF control (see, for example, \cite{Carmona,Liz2010}), which is a particular case of \eqref{TOC_eq} with $T=0$, has the form

It is also natural to consider its `external' modification
\begin{equation*}
%\label{MPF_eq}
x_{n+1}=cf(x_n ), \quad  c\in [0,1).
\end{equation*}
We recall that, once we choose an equilibrium point as the target $T$ in \eqref{TOC_eq},
this is also a fixed point of \eqref{TOC_eq}. However, PF control shifts equilibrium points (as well as TOC with a target not at an equilibrium point). Moreover, for a scalar PF control existence of a positive equilibrium is not guaranteed, for an unstable equilibrium 
there is usually a range of parameters $(c_*,c^*)$ such that a new positive equilibrium exists and is stable for $0< c_* < c < c^* < 1$ only.

\section*{Acknowledgment}

The authors are very grateful to anonymous reviewers whose valuable comments significantly contributed to the presentation of our results and the current form of the manuscript.
The first author was  supported by NSERC grant RGPIN-2020-03934.

%%%%%%%%%%%%%%%%%%%%%%%%%%%

\end{document}